\documentclass[12pt,a4paper,reqno]{amsart}

\usepackage[T1]{fontenc}
\usepackage[utf8]{inputenc}
\usepackage[english]{babel}

\usepackage[%
	left=2.5cm,       
	right=2.5cm,      
	top=3.5cm,        
	bottom=3.5cm,     
	heightrounded,
	bindingoffset=0mm 
]{geometry}

\usepackage{amsmath}
\usepackage{amsfonts}
\usepackage{amssymb}
\usepackage{lmodern}
\usepackage{mathrsfs}
\usepackage{mathtools}

\numberwithin{equation}{section}

\usepackage{hyperref}
\usepackage{xcolor}

\usepackage[noabbrev,capitalize]{cleveref}

\usepackage{autonum}

\usepackage{bookmark}

\usepackage[initials,
]{amsrefs}
\renewcommand{\MR}[1]{} 

\usepackage{enumitem}

\theoremstyle{plain}
\newtheorem{theorem}{Theorem}[section]
\newtheorem*{theorem*}{Theorem}

\newtheorem{corollary}[theorem]{Corollary}
\newtheorem{lemma}[theorem]{Lemma}

\theoremstyle{definition}
\newtheorem{definition}[theorem]{Definition}
\newtheorem{remark}[theorem]{Remark}

\newcommand{\R}{\mathbb{R}}
\newcommand{\N}{\mathbb{N}}
\newcommand{\distr}{\mathscr{D}}
\newcommand{\mm}{\mathsf{m}}
\newcommand{\dd}{{\mathsf d}}
\newcommand{\ww}{{\mathsf W}}
\newcommand{\cd}{\mathsf{CD}}
\newcommand{\rcd}{\mathsf{RCD}}
\newcommand{\mcp}{\mathsf{MCP}}
\newcommand{\be}{\mathsf{BE}}
\newcommand{\bm}{\mathsf{BM}}
\newcommand{\vfam}{\mathscr{F}}
\newcommand{\eps}{\varepsilon}
\newcommand{\prob}{{\mathscr P}}
\newcommand{\ent}{{\mathsf{Ent}}}

\newcommand{\slope}{{\mathsf D}}
\newcommand{\numb}{L} 	

\DeclareMathOperator{\vol}{vol}
\DeclareMathOperator{\spn}{span}
\DeclareMathOperator{\supp}{supp}
\DeclareMathOperator{\diverg}{div}
\DeclareMathOperator{\resto}{\mathscr R}
\DeclareMathOperator{\di}{d\!}

\DeclareMathOperator{\Lip}{Lip}
\DeclareMathOperator{\sob}{W}
\DeclareMathOperator{\sobh}{HW}
\DeclareMathOperator{\g}{g}
\DeclareMathOperator{\leb}{L}
\DeclareMathOperator{\dil}{dil}

\DeclarePairedDelimiter{\set}{\{}{\}}
\DeclarePairedDelimiter{\scalar}{<}{>}

\allowdisplaybreaks

\begin{document}

\title[Failure of CD conditions on sub-Riemannian manifolds]{Failure of curvature-dimension conditions on sub-Riemannian manifolds via tangent isometries}

\author[L.~Rizzi]{Luca Rizzi}
\address[L.~Rizzi]{Scuola Internazionale Superiore di Studi Avanzati (SISSA), via Bonomea~265, 34136 Trieste (TS), Italy}
\email{lrizzi@sissa.it}

\author[G.~Stefani]{Giorgio Stefani}
\address[G.~Stefani]{Scuola Internazionale Superiore di Studi Avanzati (SISSA), via Bonomea~265, 34136 Trieste (TS), Italy}
\email{gstefani@sissa.it {\normalfont or} giorgio.stefani.math@gmail.com}



\date{\today}

\begin{abstract}
We prove that, on any sub-Riemannian manifold endowed with a positive smooth measure, the Bakry--Émery inequality for the corresponding sub-Laplacian,
\begin{equation}
\frac{1}{2}\Delta(\|\nabla u\|^2) \geq \g(\nabla u,\nabla \Delta u) + K \|\nabla u \|^2,
\quad
K\in\R,
\end{equation}
implies the existence of enough Killing vector fields on the tangent cone to force the latter to be Euclidean at each point, yielding the failure of the curvature-dimension condition in full generality.
Our approach does not apply to non-strictly-positive measures. 
In fact, we prove that the weighted Grushin plane does not satisfy any curvature-dimension condition, but, nevertheless, does admit an a.e.\ pointwise version of the Bakry--Émery inequality.
As recently observed by Pan and Montgomery, one half of the weighted Grushin plane satisfies the $\rcd(0,N)$ condition, yielding a counterexample to gluing theorems in the $\rcd$ setting.
\end{abstract}

\maketitle

\section{Introduction and statements}

In the last twenty years, there has been an impressive effort in extending the concept of `Ricci curvature lower bound' to non-Riemannian structures, and even to general metric spaces equipped with a measure (metric-measure spaces, for short). 
We refer the reader to the ICM notes~\cite{Ambrosio-ICM} for a survey of this line of research.

There are two distinct points of view on the matter, traditionally known as the \emph{Lagrangian} and \emph{Eulerian} approaches, respectively.

The Lagrangian point of view is the one adopted by Lott--Villani and Sturm \cites{LV09,S06-I,S06-II}. 
In this formulation, Ricci curvature lower bounds are encoded by convexity-type inequalities for entropy functionals on the Wasserstein space.
Such inequalities are called \emph{curvature-dimension} conditions, $\cd(K,N)$ for short, where $K\in \R$ represents the lower bound on the curvature and $N\in [1,\infty]$ stands for an upper bound on the dimension.

The Eulerian point of view, instead, employs the metric-measure structure to define an energy form and, in turn, an associated diffusion operator. 
The notion of Ricci curvature lower bound is therefore encoded in the so-called \emph{Bakry--Émery inequality}, $\be(K,N)$ for short, for the diffusion operator, which can be expressed in terms of a suitable \emph{Gamma calculus}, see the monograph~\cite{BGL14}.

Thanks to several key contributions~\cites{AGS15,EKS15,AGS14-m,AGMR15}, the Lagrangian and the Eulerian approaches are now known to be essentially equivalent. In particular,
$\cd(K,N)$ always implies $\be(K,N)$ in \emph{infinitesimal Hilbertian} metric-measure spaces, as introduced in~\cite{G15}, while the converse implication requires further technical assumptions.

Such synthetic theory of curvature-dimension conditions, besides being consistent with the classical notions of Ricci curvature and dimension on smooth Riemannian manifolds, is stable under pointed-measure Gromov--Hausdorff convergence.
Furthermore, it yields a comprehensive approach for establishing all results typically associated with Ricci curvature lower bounds, like Poincaré, Sobolev, log-Sobolev and Gaussian isoperimetric inequalities, as well as Brunn--Minkowski, Bishop--Gromov and Bonnet--Myers inequalities.

\subsection{The sub-Riemannian framework}

Although the aforementioned synthetic cur\-va\-tu\-re-di\-men\-sion conditions embed a large variety of metric-measure spaces, a relevant and widely-studied class of smooth structures is left out---the family of \emph{sub-Riemmanian manifolds}.  
A sub-Riemannian structure is a natural generalization of a Riemannian one, in the sense that its distance is induced by a scalar product that is defined only on a smooth sub-bundle of the tangent bundle, whose rank possibly varies along the manifold. 
See the monographs~\cites{ABB20,R14,M02} for a detailed presentation.

The first result in this direction was obtained by Driver--Melcher~\cite{DM05}, who proved that an integrated version of the $\be(K,\infty)$, the so-called \emph{pointwise gradient estimate} for the heat flow, is false for the three-dimensional \emph{Heisenberg group}.

In~\cite{J09}, Juillet proved the failure of the $\cd(K,\infty)$ property for all Heisenberg groups (and even for the strictly related \emph{Grushin plane}, see~\cites{J-proceeding}).
Later, Juillet~\cite{J21} extended his result to any sub-Riemannian manifold endowed with a possibly rank-varying distribution of  rank \emph{strictly} smaller than the manifold's dimension, and with any positive smooth measure, by exploiting the notion of \emph{ample curves} introduced in \cite{ABR-variational}.
The idea of~\cites{J09,J21} is to construct a counterexample to the \emph{Brunn--Minkowski inequality}.

The `no-$\cd$ theorem' of~\cite{J09} was extended to all \emph{Carnot groups} by Ambrosio and the second-named author in~\cite{AS20}*{Prop.~3.6} with a completely different technique, namely, by exploiting the optimal version of the \emph{reverse Poincaré inequality} obtained in~\cite{BB16}.

In the case of sub-Riemannian manifolds endowed with an \emph{equiregular distribution} and a positive smooth measure, Huang--Sun~\cite{HS20} proved the failure of the $\cd(K,N)$ condition for all values of $K\in\R$ and $N\in(1,\infty)$ contradicting a bi-Lipschitz embedding result.

Very recently, in order to address the structures left out in~\cite{J21}, Magnabosco--Rossi~\cite{MR22} recently extended the `no-$\cd$ theorem' to  \emph{al\-most\--\-Rie\-man\-nian manifolds}~$M$ of dimension~$2$ or \emph{strongly regular}.  
The approach of~\cite{MR22}  relies on the localization technique developed by Cavalletti--Mondino~\cite{CM17} in metric-measure spaces.

To complete the picture, we mention that several replacements for the Lott--Sturm--Villani curvature-dimension property have been proposed and studied in the sub-Rie\-man\-nian framework in recent years.
Far from being complete, we refer the reader to~\cites{BMR22,BKS18,BKS19,BR19,BR20,M21} for an account on the Lagrangian approach, to~\cite{BG17} concerning the Eulerian one, and finally to~\cite{S22} for a first link between entropic inequalities and contraction properties of the heat flow in the special setting of metric-measure groups.

\subsection*{Main aim} 
At the present stage, a `no-$\cd$ theorem' for sub-Riemannian structures in full generality is missing, since the aforementioned approaches~\cites{DM05,HS20,J09,J21,AS20,MR22} either require the ambient space to satisfy some structural assumptions, or leave out the infinite dimensional case $N=\infty$.

The main aim of the present paper is to fill this gap by showing that (possibly rank-varying) sub-Riemannian manifolds do not satisfy any curvature bound in the sense of Lott--Sturm--Villani or Bakry--Émery when equipped with a positive smooth measure, i.e., a Radon measure whose density in local charts with respect to the Lebesgue measure is a strictly positive smooth function.

\subsection{Failure of the Bakry\texorpdfstring{--}{-}Émery inequality}

The starting point of our strategy is the weakest curvature-dimension condition, as we now define.

\begin{definition}[Bakry--Émery inequality] We say that a sub-Riemannian manifold $(M,\dd)$ endowed with a positive smooth measure $\mm$ satisfies the \emph{Bakry--Émery $\be(K,\infty)$ inequality}, for $K\in \R$, if 
\begin{equation}
\label{eq:be}
\frac12\,\Delta(\|\nabla u\|^2)
\ge 
\g(\nabla u, \nabla \Delta u)
+
K\|\nabla u\|^2 
\quad 
\text{for all}\ u \in C^\infty(M),
\end{equation}
where $\Delta$ is the corresponding sub-Laplacian, and $\nabla$ the sub-Riemannian gradient.
\end{definition}

Our first main result is the following rigidity property for sub-Riemannian structures supporting the Bakry--Émery inequality~\eqref{eq:be}.

\begin{theorem}[no-BE]
\label{res:no-BE}
Let $(M,\dd)$ be a complete sub-Riemannian manifold endowed with a positive smooth measure $\mm$.
If $(M,\dd,\mm)$ satisfies the $\be(K,\infty)$ inequality for some $K\in\R$, then $\mathrm{rank}\,\distr_x=\dim M$ at each $x\in M$, so that $(M,\dd)$ is Riemannian.
\end{theorem}

The idea behind our proof of \cref{res:no-BE} is to show that the \emph{metric tangent cone} in the sense of Gromov~\cite{G81} at each point of $(M,\dd)$ is Euclidean.
This line of thought is somehow reminiscent of the deep structural result for $\rcd(K,N)$ spaces, with $K\in\R$ and $N\in(1,\infty)$, proved by Mondino--Naber~\cite{MN19}.
However, differently from~\cite{MN19}, \cref{res:no-BE} provides information about the metric tangent cone at \emph{each} point of the manifold. Showing that the distribution $\distr$ is Riemannian at \emph{almost every} point in fact would not be enough, as this would not rule out almost-Riemannian structures. 

Starting from~\eqref{eq:be}, we first blow-up the sub-Riemannian structure and pass to its metric-measure tangent cone, showing that \eqref{eq:be} is preserved with $K=0$.
Note that, in this blow-up procedure, the positivity of the density of $\mm$ is crucial, since otherwise the resulting metric tangent cone would be endowed with the null measure.

The resulting blown-up sub-Riemannian space is isometric to a homogeneous space of the form $G/H$, where $G=\exp\mathfrak g$ is the Carnot group associated to the underlying (finite-dimensional and stratified) Lie algebra $\mathfrak g$ of bracket-generating vector fields, and $H=\exp\mathfrak h$ is its subgroup corresponding to the Lie subalgebra $\mathfrak h$ of vector fields vanishing at the origin, see~\cite{B96}. Of course, the most difficult case is when $H$ is non-trivial, that is, the tangent cone is not a Carnot group. 

At this point, the key idea is to show that the Bakry--\'Emery inequality $\be(K,\infty)$ implies the existence of special isometries on the tangent cone.

\begin{definition}[Sub-Riemannian isometries]
Let $M$ be a sub-Riemannian manifold, with distribution $\distr$ and metric $\g$. A diffeomorphism $\phi:M \to M$ is an \emph{isometry} if
\begin{equation}\label{eq:point-dist-pres}
(\phi_*\distr)|_x = \distr_{\phi(x)} 
\quad 
\text{for all}\ x\in M,
\end{equation}
and, furthermore, $\phi_*$ is an orthogonal map with respect to $\g$. We say that a smooth vector field $V$ is \emph{Killing} if its flow $\phi_t^V$ is an isometry for all $t\in \R$.
\end{definition}

For precise definitions of $\mathfrak{g}$ and $\mathfrak{h}$ in the next statement, we refer to \cref{subsec:nilpotent_approx}.

\begin{theorem}[Existence of Killing fields]\label{res:space_i}
Let $(M,\dd)$ be a complete sub-Riemannian manifold equipped with a positive smooth measure $\mm$
If $(M,\dd,\mm)$ satisfies the $\be(K,\infty)$ inequality for some $K\in \R$, then, for the nilpotent approximation at any given point, there exists a vector space $\mathfrak{i}\subset \mathfrak{g}^{1}$ such that
\begin{equation}
\label{eq:space_i_complement}
\mathfrak{g}^{1} = \mathfrak{i} \oplus \mathfrak{h}^{1}
\end{equation}
and every $Y \in \mathfrak{i}$ is a Killing vector field.
\end{theorem}

The existence of the space of isometries $\mathfrak{i}$ forces the Lie algebra $\mathfrak{g}$ to be commutative and of maximal rank, thus implying that the original manifold $(M,\dd)$ was in fact Riemannian.

\begin{theorem}[Killing implies commutativity]
\label{res:commutativity}
If there exists a subspace $\mathfrak{i}\subset \mathfrak{g}^{1}$ of Killing vector fields such that $\mathfrak{g}^1 = \mathfrak{i} \oplus \mathfrak{h}^1$, then $\mathfrak{g}$ is commutative.
\end{theorem}

\cref{res:commutativity} states that, if a Carnot group contains enough \emph{horizontal symmetries}, then it must be commutative. 
As it will be evident from its proof, \cref{res:commutativity} holds simply assuming that, for each $V\in \mathfrak{i}$,  the flow $\phi^V_t$ is pointwise distribution-preserving, namely it satisfies \eqref{eq:point-dist-pres}, without being necessarily isometries.
 
\subsection{Infinitesimal Hilbertianity}

The Bakry--\'Emery inequality $\be(K,\infty)$ in~\eqref{eq:be} is a consequence of the $\cd(K,\infty)$ condition as soon as the ambient metric-measure space is infinitesimal Hilbertian as defined in \cite{G15}.

Let $(X,\dd)$ be a complete separable metric space, $\mm$ be a locally bounded Borel measure, and $q\in [1,\infty)$. We let $|\slope u|_{w,q}\in\leb^q(X,\mm)$ be the \emph{minimal $q$-upper gradient} of a measurable function $u: X \to \R$, see~\cite{AGS13}*{Sec.~4.4}. 
We define the Banach space
\begin{equation}
\sob^{1,q}(X,\dd,\mm)
=
\set*{u\in\leb^q(X,\mm) : |\slope u|_{w,q} \in \leb^q(X,\mm)}
\end{equation}
with the norm
\begin{equation}
\label{eq:cheeger_norm}
\|u\|_{\sob^{1,q}(X,\dd,\mm)}=\left(\|u\|_{\leb^q(X,\mm)}^q+\||\slope u|_{w,q}\|_{\leb^q(X,\mm)}^q\right)^{1/q}.
\end{equation}

\begin{definition}[Infinitesimal Hilbertianity]
A metric measure space $(X,\dd,\mm)$ is \emph{infinitesimally Hilbertian} if $\sob^{1,2}(X,\dd,\mm)$ is a Hilbert space.
\end{definition}

The infinitesimal Hilbertianity of sub-Riemannian structures has been recently proved in~\cite{LLP22}, with respect to any Radon measure.
In particular, \cref{res:no-BE} immediately yields the following `no-$\cd$ theorem' for sub-Riemannian manifolds, thus extending all the aforementioned results~\cites{DM05,HS20,J09,J21,AS20,MR22}. 

\begin{corollary}[no-$\cd$]\label{res:no-CD}
Let $(M,\dd)$ be a complete sub-Riemannian manifold endowed with a positive smooth measure $\mm$.
If $(M,\dd,\mm)$ satisfies the $\cd(K,\infty)$ condition for some $K\in \R$, then $(M,\dd)$ is Riemannian.
\end{corollary}
 
However, since the measure in \cref{res:no-CD} is positive and smooth, we can avoid to rely on the general result of~\cite{LLP22}, instead providing a simpler and self-contained proof of the infinitesimal Hilbertianity property. In particular, we prove the following result, which actually refines~\cite{LLP22}*{Th.~5.6} in the case of smooth measures.
In the following, $\sobh^{1,q}(M,\mm)$ denotes the sub-Riemannian Sobolev spaces (see \cref{sec:grad}).

\begin{theorem}[Infinitesimal Hilbertianity]\label{res:R}
Let $q\in(1,\infty)$. 
Let $(M,\dd)$ be a complete sub-Riemannian manifold equipped with a positive smooth measure $\mm$.
The following hold. 
\begin{enumerate}[label=(\roman*),ref=\roman*,itemsep=.5ex]

\item \label{item:sob=sob} 
$\sob^{1,q}(M,\dd,\mm)=\sobh^{1,q}(M,\mm)$, with 
$|\slope u|_{w,q}=\|\nabla u \|$ $\mm$-a.e.\ on $M$ for all $u\in\sob^{1,q}(M,\dd,\mm)$.
In particular, taking $q=2$, $(M,\dd,\mm)$ is infinitesimally Hilbertian.

\item\label{item:bochner}
If $(M,\dd,\mm)$ satisfies the $\mathsf{CD}(K,\infty)$ condition for some $K\in\R$, then the Bakry--Émery $\be(K,\infty)$ inequality \eqref{eq:be} holds on $M$.
\end{enumerate}
\end{theorem}

Note that \cref{res:R} holds for less regular measures, see \cref{rmk:moregeneralmeasures}. 

\begin{remark}[The case of a.e.\ smooth measures]\label{rmk:inf_hilb_local}
\cref{res:R} can be adapted also to the case of a Borel and locally finite measure $\mm$ which is smooth and positive only on $\overline{\Omega}$, where $\Omega \subset M$ is an open set with $\mm(\partial\Omega)=0$. In this case, we obtain $\sobh^{1,q}(\Omega,\mm)= \sob^{1,q}(\overline{\Omega},\dd,\mm)$, with $|\slope u|_{w,q}=\|\nabla u \|$  $\mm$-a.e.\ on $\Omega$ for all $u\in \sob^{1,q}(\overline{\Omega},\dd,\mm)$.
In particular, if $\mm$ is smooth and positive out of a closed set $\mathcal{Z}$, with $\mm(\mathcal{Z})=0$, an elementary approximation argument proves that $(M,\dd,\mm)$ is infinitesimally Hilbertian and, if $(M,\dd,\mm)$ satisfies the $\cd(K,\infty)$ condition for $K\in \R$, then the Bakry-Émery $\be(K,\infty)$ inequality \eqref{eq:be} holds on $M\setminus \mathcal{Z}$. This is the case, for example, of the Grushin planes and half-planes with weighted measures of \cref{sec:weigthedgrush}. The proof follows the same argument of the one of \cref{res:R}, exploiting the locality of the $q$-upper gradient, see for example~\cite{AGS13}*{Sec.~8.2} and~\cite{G15}*{Prop.~2.6}, and similar properties for the distributional derivative.
\end{remark}

\begin{remark}[A comment on Juillet's no-CD theorem] 
As we already mentioned, in~\cite{J21} Juillet proved that, if a sub-Riemannian distribution on an $n$-dimensional manifold has rank strictly smaller than $n$, then the corresponding metric measure space cannot satisfy the $\cd(K,\infty)$ condition. In other words, $\cd(K,\infty)$ implies that the sub-Riemannian distribution has maximal rank at some point. Actually, Juillet's argument is local, and thus $\cd(K,\infty)$ implies that the distribution must be of maximal rank on a dense set of points. Our method (\cref{res:no-CD}) shows that this is actually the case for \emph{all} points.
\end{remark}

\subsection{An alternative approach to the \texorpdfstring{`no-$\cd$ theorem'}{no-CD theorem}}
We mention an alternative proof of the `no-$\cd$ theorem' for almost-Riemannian structures (i.e., sub-Riemannian structures that are Riemannian outside a closed nowhere dense singular set). 
The strategy relies on the Gromov-Hausdorff continuity of the metric tangent at interior points of geodesics in $\rcd(K,N)$ spaces, with $N<\infty$, proved by Deng in~\cite{D22}, 

For example, consider the standard Grushin plane (introduced in \cref{sec:weigthedgrush}) equipped with a smooth positive measure. The curve $\gamma(t)=(t,0)$, $t\in\R$, is a geodesic between any two of its point. The metric tangent at $\gamma(t)$ is (isometric to) the Euclidean plane for every $t\ne0$, while it is (isometric to) the Grushin plane itself for $t=0$. 
However the Euclidean plane is not isometric to the Grushin plane, contradicting the continuity result.

This strategy has a few drawbacks. 
On the one hand, it relies on the (non-trivial) machinery developed in \cite{D22}.
Consequently, this argument does not work in the case \mbox{$N=\infty$}. 
On the other hand, the formalization of this strategy for general al\-most-Rie\-man\-ni\-an structures requires certain non-isometric-embedding results into Euclidean spaces, which we are able to prove only under the same assumptions of~\cite{MR22}.

\subsection{Weighted Grushin structures}\label{sec:weigthedgrush}

When the density of the smooth measure is allowed to vanish, the `no-$\cd$ theorem' breaks down.
In fact, in this situation, the following two interesting phenomena occur:

\begin{enumerate}[leftmargin=2.5em,rightmargin=1em,itemsep=1ex,topsep=1ex,label=(\Alph*),ref=\Alph*]

\item\label{item:be_no_cd}
the Bakry-\'Emery $\be(K,\infty)$ inequality no longer implies the $\cd(K,\infty)$ condition;

\item\label{item:boundary} 
there exist al\-most-Rie\-man\-nian structures with boundary satisfying the $\cd(0,N)$ condition for $N \in[1,\infty]$.
\end{enumerate}

We provide examples of both phenomena on the so-called \emph{weighted Grushin plane}. 
This is the sub-Riemannian structure on $\R^2$ induced by the family $\vfam =\{X,Y\}$, where
\begin{equation}\label{eq:grushin_frame}
X=\partial_x,
\quad
Y=x\,\partial_y,
\quad
(x,y)\in\R^2.
\end{equation}
The induced distribution $\distr=\spn\set*{X,Y}$ has maximal rank outside the singular region $S=\set*{x=0}$ and rank $1$ on $S$. 
Since $[X,Y]=\partial_y$ on $\R^2$, the resulting sub-Riemannian metric space $(\R^2,\dd)$ is Polish and geodesic. 
It is \emph{almost-Riemannian} in the sense that, out of $S$, the metric is locally equivalent to the Riemannian one given by the metric tensor
\begin{equation}\label{eq:grushin_rm}
\g=\di x\otimes\di x+\frac1{x^2}\,\di y\otimes\di y,
\quad
x\ne0.
\end{equation}
We endow the metric space $(\R^2,\dd)$ with the weighted Lebesgue measure
\begin{equation}
\label{eq:grushin_mm}
\mm_p=|x|^p\di x\di y,
\end{equation}
where $p\in\R$ is a parameter.
The choice $p=-1$ corresponds to the Riemannian density
\begin{equation}
\label{eq:grushin_vol_g}
\vol_{\g}=\frac{1}{|x|}\,\di x\,\di y,
\quad
x\ne0,
\end{equation}
so that 
\begin{equation}
\label{eq:grushin_wmeas_p}
\mm_p=e^{-V}\vol_{\g}, \quad V(x)=-(p+1)\log |x|,
\quad
x\ne0.
\end{equation}
We call the metric-measure space $\mathbb{G}_p=(\R^2,\dd,\mm_p)$ the \emph{($p$-)weighted Grushin plane}.

We can now state the following result, illustrating phenomenon~\eqref{item:be_no_cd}.

\begin{theorem}
\label{res:fullgrushin}
Let $p\in\R$ and let $\mathbb{G}_p=(\R^2,\dd,\mm_p)$ be the weighted Grushin plane.
\begin{enumerate}[label=(\roman*),ref=\roman*,itemsep=.5ex]

\item\label{res:juillet}
If $p\ge0$, then $\mathbb{G}_p$ does not satisfy the $\cd(K,\infty)$ property for all $K\in\R$.

\item\label{res:grushinbochnerae}
If $p\ge1$, then $\mathbb{G}_p$ satisfies the $\be(0,\infty)$ inequality \eqref{eq:be} almost everywhere. 
\end{enumerate}
\end{theorem}

To prove~\eqref{res:juillet}, we show that the corresponding Brunn--Minkowski inequality is violated.
In fact, the case $p=0$ is due to Juillet~\cites{J-proceeding}, while the case $p>0$ can be achieved via a simple argument which was pointed out to us by J.~Pan. Claim~\eqref{res:grushinbochnerae}, instead, is obtained by direct computations.

Somewhat surprisingly, the weighted Grushin \emph{half}-plane $\mathbb{G}_p^+$---obtained by restricting the metric-measure structure of $\mathbb{G}_p$ to the (closed) half-plane $[0,\infty)\times\R$---does satisfy the $\cd(0,N)$ condition for sufficiently large $N\in[1,\infty]$.
Precisely, we can prove the following result, illustrating phenomenon~\eqref{item:boundary}.

\begin{theorem}\label{res:halfgrushinCD}
Let $p\geq 1$. The weighted Grushin half-plane $\mathbb{G}_p^+ $ satisfies the $\cd(0,N)$ condition if and only if $N\geq N_p$, where $N_p \in (2,\infty]$ is given by
\begin{equation}
\label{eq:halfgrushin_Np}
N_p = \frac{(p+1)^2}{p-1}+2,
\end{equation}
with the convention that $N_1 = \infty$. 
Furthermore, $\mathbb{G}_p^+ $ is infinitesimally Hilbertian, and it is thus an $\rcd(0,N)$ space for $N\geq N_p$.
\end{theorem}

While we were completing this work, Pan and Montgomery~\cite{P22} observed that the family of Ricci limits built in~\cites{PW22,DHPW22} include the weighted Grushin half-spaces presented above as special cases. 
Our construction and method of proof are more direct with respect to the approach of \cites{PW22,DHPW22}, and easily yield sharp dimensional bounds.

\subsection{Counterexample to gluing theorems}

We end this introduction with an interesting by-product of our analysis, in in connection with the so-called \emph{gluing theorems}.

Perelman's Doubling Theorem \cite{Perelman}*{Sect.~5.2} states that a finite dimensional Alexandrov space with a curvature lower bound can be doubled along its boundary yielding an Alexandrov space with \emph{same} curvature lower bound and dimension. This result has been extended by Petrunin~\cite{Petrunin}*{Thm.~2.1} to the gluing of Alexandrov spaces. 

It is interesting to understand whether these classical results hold true for general metric-measure spaces satisfying synthetic Ricci curvature lower bounds in the sense of Lott--Sturm--Villani. In \cite{KKS20-gluing}, the gluing theorem was proved for $\cd(K,N)$ spaces with Alexandrov curvature bounded from below (while it is false for $\mcp$ spaces, see~\cite{R18-gluing}).

Here we obtain that, in general, the assumption of Alexandrov curvature bounded from below cannot be removed from the results in~\cite{KKS20-gluing}. More precisely, \cref{res:fullgrushin,res:halfgrushinCD}, and the fact that the metric-measure double of the Grushin half-plane $\mathbb{G}_p^+$ is $\mathbb{G}_p$ (see \cite{R18-gluing}*{Prop.~6}) yield the following corollary.

\begin{corollary}[Counterexample to gluing in $\rcd$ spaces]\label{res:counterexample}
For all $N\geq 10$, there exists a geodesically convex $\rcd(0,N)$ metric-measure space with boundary such that its metric-measure double does not satisfy the $\cd(K,\infty)$ condition for any $K\in \R$.
\end{corollary}

In \cite{KKS20-gluing}*{Conj.~1.6}, the authors conjecture the validity of the gluing theorem for \emph{non-collapsed} $\rcd(K,N)$, with $N$ the Hausdorff dimension of the metric-measure space. 
As introduced in~\cite{Gigli-de-PhilippisJEP}, a non-collapsed $\rcd(K,N)$ space is an infinitesimally Hilbertian $\cd(K,N)$ space with $\mm = \mathscr{H}^N$, where $\mathscr{H}^N$ denotes the $N$-dimensional Hausdorff measure of $(X,\dd)$. Since the weighted half-Grushin spaces are indeed collapsed, \cref{res:counterexample} also shows that the non-collapsing assumption cannot be removed from~\cite{KKS20-gluing}*{Conj.~1.6}.

\subsection{Acknowledgments}

We wish to thank Michel Bonnefont for fruitful discussions and, in particular, for bringing some technical details in \cite{DM05} that inspired the strategy of the proof of \cref{res:no-BE} to our attention.

This work has received funding from the European Research Council (ERC) under the European Union’s Horizon 2020 research and innovation programme (grant agreement No.\ 945655) and the ANR grant `RAGE' (ANR-18-CE40-0012). 

The second-named author is  member of the Istituto Nazionale di Alta Matematica (INdAM), Gruppo Nazionale per l'Analisi Matematica, la Probabilità e le loro Applicazioni (GNAMPA), and is partially supported by the INdAM--GNAMPA 2022 Project \textit{Analisi geometrica in strutture subriemanniane}, codice CUP\_E55\-F22\-000\-270\-001, and by the INdAM--GNAMPA 2023 Project \textit{Problemi variazionali per funzionali e operatori non-locali}, codice CUP\_E53\-C22\-001\-930\-001. 

\section{Preliminaries}
\label{sec:preliminaries}

In this section, we introduce some notation and recall some results about sub-Ri\-e\-man\-nian manifolds and curvature-dimension conditions.

\subsection{Sub-Riemannian structures}
\label{subsec:sub-R}

For $\numb\in\N$, we let $\vfam=\set*{X_1,\dots,X_{\numb}}$ be a family of smooth vector fields globally defined on a smooth $n$-dimensional manifold~$M$, $n\geq 2$.
The (generalized) \emph{sub-Riemannian distribution} induced by the family $\vfam$ is defined by
\begin{equation}
\label{eq:def_distr}
\distr = \bigsqcup_{x\in M} \distr_x,
\quad 
\distr_x = \spn\set*{X_1|_x,\dots,X_{\numb}|_x} 
\subset 
T_x M,\quad
x\in M.
\end{equation}
Note that we do not require  the dimension of $\distr_x$ to be constant as $x\in M$ varies, that is, we may consider \emph{rank-varying} distributions. With a standard abuse of notation, we let  
\begin{equation}
\Gamma(\distr)
=
\text{$C^\infty$-module generated by $\vfam$}.
\end{equation}
Notice that, for any smooth vector field $V$, it holds
\begin{equation}
V\in \Gamma(\distr)
\implies
V_x \in \distr_x\ \text{for all}\ x\in M,
\end{equation}
but the converse is false in general. We let
\begin{equation}
\label{eq:def_norm}
\|V\|_x = \min\set*{|u| :u\in \R^{\numb}\ \text{such that}\ V = \sum_{i=1}^{\numb} u_i\, X_i|_x,\ X_i\in\vfam}
\end{equation}
whenever $V\in\distr$ and $x\in M$.
The \emph{norm} $\| \cdot\|_x$ induced by the family $\vfam$ satisfies the \emph{parallelogram law} and, consequently, it is induced by a \emph{scalar product}
\begin{equation}
\g_x \colon \distr_x \times \distr_x \to \R.
\end{equation}

An \emph{admissible curve} is a locally Lipschitz in charts path $\gamma\colon[0,1]\to M$ such that there exists a \emph{control} $u \in \leb^\infty([0,1];\R^{\numb})$ such that
\begin{equation}
\dot\gamma(t) = \sum_{i=1}^{\numb} u_i(t) X_i |_{\gamma(t)} \quad \text{for a.e.}\ t \in [0,1].
\end{equation}
The \emph{length} of an admissible curve $\gamma$ is defined via the norm~\eqref{eq:def_norm} as 
\begin{equation}
\mathrm{length}(\gamma)
=
\int_0^1\|\dot\gamma(t)\|_{\gamma(t)}\di t
\end{equation}
and the \emph{Carnot--Carathéodory} (or \emph{sub-Riemannian}) \emph{distance}  between $x,y\in M$ is
\begin{equation}
\label{eq:def_distance}
\dd(x,y)
=
\inf\set*{\mathrm{length}(\gamma) : \gamma\ \text{admissible with}\ \gamma(0)=x,\ \gamma(1)=y}.
\end{equation}  

We assume that the family $\vfam$ satisfies the \emph{bracket-generating condition} 
\begin{equation}
\label{eq:hormander}
T_xM=\set*{X|_x : X\in\mathrm{Lie}(\vfam)}
\quad
\text{for all}\ x\in M,
\end{equation}
where $\mathrm{Lie}(\vfam)$ is the smallest Lie subalgebra of vector fields on $M$ containing $\vfam$, namely,
\begin{equation}
\mathrm{Lie}(\vfam)
=
\spn\set*{[X_{i_1},\dots,[X_{i_{j-1}},X_{i_j}]] : X_{i_\ell}\in\vfam, \ j\in\N}.
\end{equation}
Under the assumption~\eqref{eq:hormander}, the Chow--Rashevskii Theorem implies that $\dd$ is a well-defined finite distance on $M$ inducing the same topology of the ambient manifold.

\subsection{Gradient, sub-Laplacian and Sobolev spaces}
\label{sec:grad}

The \emph{gradient} of a function $u\in C^\infty(M)$ is the unique vector field $\nabla u \in \Gamma(\distr)$ such that
\begin{equation}
\label{eq:def_gradient}
\g(\nabla u, V) = du(V)
\quad
\text{for all}\ V \in \Gamma(\distr).
\end{equation}
One can check that $\nabla u$ can be globally represented as
\begin{equation}
\label{eq:gradient_representation}
\nabla u = \sum_{i=1}^{\numb} X_iu\,X_i, 
\quad
\text{with} 
\quad 
\|\nabla u\|^2 = \sum_{i=1}^{\numb} (X_i u)^2,
\end{equation}
even if the family $\vfam$ is not linearly independent, see \cref{res:formule_utili} for a proof.

We equip the manifold~$M$ with a \emph{positive smooth} measure $\mm$. 
The \emph{sub-Laplacian} of a function $u\in C^\infty(M)$ is the unique function $\Delta u\in C^\infty(M)$ such that 
\begin{equation}
\label{eq:def_laplacian}
\int_M\g(\nabla u,\nabla v) \di \mm = -\int_M v\,\Delta u\di \mm
\end{equation}
for all $v\in C^\infty_c(M)$.
On can check that $\Delta u$ can be globally represented as 
\begin{equation}
\label{eq:laplacian_representation}
\Delta u = \sum_{i=1}^{\numb} \left(X_i^2u + X_iu\,\diverg_\mm (X_i)\right),
\end{equation}
see \cref{res:formule_utili} for a proof. 
In \eqref{eq:laplacian_representation}, $\mathrm{div}_\mm V$ is the divergence of the vector field $V$ computed with respect to $\mm$, that is,
\begin{equation}
\label{eq:div_m_def}
\int_M v\,\diverg_\mm(V)\di \mm=-\int_M \g(\nabla v,V)\di \mm \quad \text{for all}\ v \in C^\infty_c(M).
\end{equation}

For $q\in[1,\infty)$, we say that $u\in \leb^1_{\mathrm{loc}}(M,\mm)$ has $q$-integrable \emph{distributional $X_i$-derivative} if there exists a function $X_i u \in \leb^q(M,\mm)$ such that
\begin{equation}
\int_M v X_i u \di \mm = \int_M  u X_i^*v \di \mm\quad 
\text{for all}\ v \in C^\infty_c(M),
\end{equation}
where $X_i^*v = -X_i v - v \diverg_{\mm} (X_i)$ denotes the adjoint action of $X_i$. 
We thus let
\begin{equation}
\sobh^{1,q}(M,\mm)
=
\set*{u\in\leb^q(M,\mm) : X_iu\in\leb^q(M,\mm),\ i=1,\dots,\numb}
\end{equation}
be the usual \emph{horizontal $\sob^{1,q}$ Sobolev space} induced by the the family $\vfam$ and the measure~$\mm$ on~$M$, endowed with the natural norm
\begin{equation}
\|u\|_{\sobh^{1,q}(M,\mm)}
=
\left(\|u\|_{\leb^q(M,\mm)}^q+\|\nabla u\|_{\leb^q(M,\mm)}^q\right)^{1/q}
\end{equation}
for all $u\in\sobh^{1,q}(M,\mm)$, 
where $\nabla u =\displaystyle\sum_{i=1}^{\numb} X_i u\,X_i$ in accordance with~\eqref{eq:gradient_representation} and 
\begin{equation}
\|\nabla u\|_{\leb^q(M,\mm)}^q = \int_M \|\nabla u\|^q \di \mm.
\end{equation}

\subsection{Privileged coordinates}
\label{subsec:privileged_coord}

Following \cites{B96,Jeanbook}, we introduce \emph{privileged coordinates}, a fundamental tool in the description of the tangent cone of sub-Riemannian manifolds. 

Given a multi-index $I\in \{1,\dots,\numb\}^{\times i}$, $i\in\N$, we let $|I| = i$ be its \emph{length} and we set 
\begin{equation}
X_I = [X_{I_1},[\dots,[X_{I_{i-1}},X_{I_i}]]]].
\end{equation}
Accordingly, we define
\begin{equation}
\label{eq:def_distr_i}
\distr_x^i = \spn\{X_I|_x : |I|\leq i\}
\end{equation}
and 
\begin{equation}
k_i(x)=\dim \distr_x^i
\end{equation}
for all $x\in M$ and $i\in\N$.
In particular, $\distr_x^0=\set*{0}$ and $\distr_x^1=\distr_x$ as in~\eqref{eq:def_distr} for all $x\in M$.
The spaces defined in~\eqref{eq:def_distr_i} naturally yield the filtration
\begin{equation}\label{eq:filtration}
\set*{0}
=
\distr_x^0 
\subset
\distr_x^1
\subset 
\dots 
\subset \distr_x^{s(x)} =T_x M
\end{equation}
for all $x\in M$, where $s=s(x)\in\N$ is the \emph{step} of the sub-Riemannian structure at the point~$x$. 
We say that $x\in M$ is a \emph{regular} point if the dimension of each space~$\distr_y^i$ remains constant as $y\in M$ varies in an open neighborhood of $x$, otherwise $x$ is a \emph{singular} point.

\begin{definition}[Adapted and privileged coordinates]
\label{def:coord}
Let $o\in M$ and let $U\subset M$ be an open neighborhood of $o$.
We say that the local coordinates given by a diffeomorphism $z\colon U\to\R^n$ are \emph{adapted at $o$} if they are \emph{centered at $o$}, i.e.\ $z(o)=0$, and $\partial_{z_1}|_0,\dots,\partial_{z_{k_i}}|_0$ form a basis for $\distr_o^i$ in these coordinates for all $i=1,\dots,s(o)$. 
We say that the adapted coordinate $z_i$ has \emph{weight} $w_i=j$ if $\partial_{z_i}|_0\in\distr_o^j\setminus \distr_o^{j-1}$.
Furthermore, we say that the coordinates $z$ are  \emph{privileged at $o$} if they are adapted at $o$ and, in addition, $z_i(x) = O(\dd(x,o)^{w_i})$ for all $x\in U$ and $i=1,\dots,n$.
\end{definition}

Privileged coordinates exist in a neighborhood of any point, see~\cite{B96}*{Thm.~4.15}. 

\subsection{Nilpotent approximation}
\label{subsec:nilpotent_approx}

From now on, we fix a set of privileged coordinates  $z\colon U\to\R^n$ around a point $o\in M$ in the sense of \cref{def:coord}.
Without loss of generality, we identify the coordinate domain $U\subset M$ with $\R^n$ and the base point $o\in M$ with the origin $0\in\R^n$. 
Similarly, the vector fields in $\vfam$ defined on $U$ are identified with vector fields on $\R^n$, and the restriction of the sub-Riemannian distance $\dd$ to $U$ is identified with a distance function on $\R^n$, which is induced by the family $\vfam$, for which we keep the same notation. 

On $(\R^n,\vfam)$, we define a family of \emph{dilations}, for $\lambda\geq 0$, by letting
\begin{equation}
\dil_\lambda \colon \R^n\to \R^n,
\quad
\dil_\lambda(z_1,\dots,z_n) = (\lambda^{w_1}z_1,\dots,\lambda^{w_n}z_n)
\end{equation}
for all $z=(z_1,\dots,z_n)\in\R^n$, where the $w_i$'s are the weights given by \cref{def:coord}. 
We say that a differential operator $P$ is \emph{homogeneous of degree $-d \in \mathbb{Z}$} if
\begin{equation}
\label{eq:def_degree_vec}
P(f\circ\dil_\lambda)=\lambda^{-d}(Pf)\circ\dil_\lambda
\quad
\text{for all}\
\lambda>0\
\text{and}\
f\in C^\infty(\R^n).
\end{equation} 
Note that the monomial $z_i$ is homogeneous of degree $w_i$, while the vector field $\partial_{z_i}$ is homogeneous of degree $-w_i$, for $i=1,\dots,n$. As a consequence, the differential operator
\begin{equation}
z_1^{\mu_1} \cdot \dots \cdot z_n^{\mu_n} \frac{\partial^{|\nu|}}{\partial z^{\nu_1}_1 \cdots \partial z^{\nu_n}_n}, \qquad \nu_i,\mu_j\in \N \cup \{0\},
\end{equation}
is homogeneous of degree $\sum_{i=1}^n w_i (\mu_i-\nu_i)$. 
For more details, see \cite{B96}*{Sec.~5}.

We can now introduce the new family
\begin{equation}
\widehat\vfam
=
\set*{\widehat X_1,\dots,\widehat X_\numb}
\end{equation}
by defining
\begin{equation}
\label{eq:def_vec_approx}
\widehat X_i
= 
\lim_{\eps\to 0}
X_i^\eps,
\quad
X_i^\eps 
=
\eps
\,(\dil_{1/\eps})_* X_i,
\end{equation}
for all $i=1,\dots,\numb$, where $(\dil_{1/\eps})_*$ stands for the usual push-forward via the differential of the dilation map $\dil_{1/\eps}$,
see~\cite{B96}*{Sec.~5.3}.
The convergence in~\eqref{eq:def_vec_approx} can be actually made more precise, in the sense that 
\begin{equation}
X_i^\eps
=
\widehat X_i
+
\resto_i^\eps,
\quad
i=1,\dots,\numb, 
\end{equation}
where $\resto_i^\eps$ locally uniformly converges to zero as $\eps\to 0$, see~\cite{B96}*{Thm.~5.19}. 

The family $\widehat\vfam$ is a set of complete vector fields on $\R^n$, homogeneous of degree $-1$, with polynomial coefficients, and can be understood as the `principal part' of $\vfam$ upon blow-up by dilations. Since $\vfam$ satisfies the bracket-generating condition~\eqref{eq:hormander}, also the new family $\widehat\vfam$ is  bracket-generating at all points of $\R^n$, and thus induces a finite sub-Riemannian distance~$\widehat\dd$, see~\cite{B96}*{Prop.~5.17}. The resulting $n$-dimensional sub-Riemannian structure $(\R^n,\widehat\vfam\,)$ is called \emph{nilpotent approximation} of $(\R^n,\vfam)$ at $0\in\R^n$.

The family $\widehat\vfam=\set*{\widehat{X}_1,\dots,\widehat{X}_{\numb}}$ generates a finite-dimensional stratified Lie algebra 
\begin{equation}
\label{eq:g_def}
\mathfrak{g}
=
\mathrm{Lie}(\widehat{\vfam}\,) = \mathfrak{g}^{1}\oplus \dots \oplus \mathfrak{g}^{s}
\end{equation}
of step $s=s(0)\in\N$, where the grading is given by the degree of the vector fields, according to the definition in~\eqref{eq:def_degree_vec}, that is, 
the layer $\mathfrak{g}^{i}$ corresponds to vector fields homogeneous of degree $-i$ with respect to dilations, see~\cite{B96}*{Sec.~5.4}.
In particular, $\mathfrak{g}^{1} = \spn\set*{\widehat{X}_1,\dots,\widehat{X}_{\numb}}$, so that $\mathfrak{g}$ is generated by its first stratum, namely,
\begin{equation}
\label{eq:ggen}
\mathfrak g^{j+1}=[\mathfrak g^1,\mathfrak g^j], \qquad \forall j=1,\dots,s-1.
\quad
\end{equation}
Finally, define the Lie subalgebra of vector fields vanishing at $0$,
\begin{equation}
\label{eq:h_def}
\mathfrak{h} = \set*{\widehat X\in \mathfrak{g}: \widehat X|_0 = 0} =  \mathfrak{h}^{1}\oplus \dots \oplus \mathfrak{h}^{s},
\end{equation}
which inherits the grading from the one of $\mathfrak{g}$,
\begin{equation}
\label{eq:hgen}
\mathfrak h^{j+1}=[\mathfrak h^1,\mathfrak h^j], \qquad \forall j=1,\dots,s-1.
\end{equation}
It is a fundamental fact~\cite{B96}*{Thm.~5.21} that the nilpotent approximation $(\R^n,\widehat\vfam\,)$ is diffeomorphic to the homogeneous sub-Ri\-e\-man\-nian space $G/H$, where $G$ is the Carnot group $G=\exp\mathfrak g$ (explicitly realized as the subgroup of the flows of the vector fields of $\mathfrak{g}$ acting on $\R^n$ from the right) and $H=\exp\mathfrak h$ is the Carnot subgroup induced by $\mathfrak{h}$.

In particular, if $0\in\R^n$ is a regular point, then $H=\set*{0}$, and so the nilpotent approximation $(\R^n,\widehat\vfam\,)$ is diffeomorphic to the Carnot group $G$, see~\cite{B96}*{Prop.~5.22}.

Recall that the smooth measure $\mm$ on the original manifold $M$ can be identified with a smooth measure on $U\simeq \R^n$, for which we keep the same notation.
In particular, $\mm$ is absolutely continuous with respect to the $n$-dimensional Lebesgue measure $\mathscr L^n$ on~$\R^n$, with $\mm=\rho\,\mathscr L^n$ for some positive smooth function $\rho\colon\R^n\to(0,\infty)$. The corresponding blow-up measure on the nilpotent approximation is naturally given by
\begin{equation}
\label{eq:def_meas_blow-up}
\widehat\mm
=
\lim_{\eps\to 0}
\mm^\eps
=
\rho(0)\,\mathscr L^n,
\quad
\mm^\eps=\eps^Q\,(\dil_{1/\eps})_\#\mm,
\end{equation}
in the sense of weak$^*$ convergence of measures in $\R^n$, where 
\begin{equation}
Q=\sum_{i=1}^n i\,w_i\in\N
\end{equation}
is the so-called homogeneous dimension of $(\R^n,\widehat\vfam\,)$ and $(\dil_{1/\eps})_\#$ stands for the push-forward in the measure-theoretic sense via the dilation map $\dil_{1/\eps}$.
Consequently, without loss of generality, we can assume that $\rho(0)=1$, thus endowing $(\R^n,\widehat\vfam\,)$ with the $n$-dimensional Lebesgue measure. 
Notice that $\mathrm{div}_{\mathscr L^n}\widehat X_i=0$, for all $i=1,\dots,\numb$, since each $\widehat X_i$ is homogeneous of degree $-1$. 
Hence, by \eqref{eq:laplacian_representation}, the sub-Laplacian of a function $u\in C^\infty(\R^n)$ can be globally represented as 
\begin{equation}\label{eq:laplaciannilpotent}
\widehat\Delta u 
= 
\sum_{i=1}^{\numb} 
\widehat X_i^{\,2}u.
\end{equation}

It is worth noticing that the metric space $(\R^n,\widehat\dd\,)$ induced by the nilpotent approximation $(\R^n,\widehat\vfam\,)$ actually coincides with the \emph{metric tangent cone} at $o\in M$ of the metric space $(M,\dd)$ in the sense of Gromov~\cite{G81}, see~\cite{B96}*{Thm.~7.36} for the precise statement.
In fact, the sub-Riemmanian distance $\dd^\eps$ induced by the vector fields $X_i^\eps$, $i=1,\dots,\numb$, defined in~\eqref{eq:def_vec_approx} is uniformly converging to the distance $\widehat{\dd}$ on compact sets as $\eps\to 0$.

It is not difficult to check that the family $\set*{(\R^n,\dd^\eps,\mm^\eps,0)}_{\eps>0}$ of pointed metric-measure spaces converge to the pointed metric-measure space $(\R^n,\widehat\dd,\mathscr L^n,0)$ as $\eps \to 0$ in the \emph{pointed measure Gromov--Hausdorff topology}, see~\cite{BMR22}*{Sec.~10.3} for details.

\subsection{The curvature-dimension condition}\label{sec:cdcondition}

We end this section by recalling the definition of cur\-va\-tu\-re\--\-di\-men\-sion conditions of introduced in \cites{S06-I,S06-II,LV09}. 

On a Polish (i.e., separable and complete) metric space  $(X,\dd)$, we let $\prob(X)$ be the set of probability Borel measures on $X$ and define the \emph{Wasserstein (extended) distance $\ww_2$}
\begin{equation}
\ww_2^2(\mu,\nu)
=
\inf\set*{\int_{X\times X} \dd^2(x,y)\di\pi : \pi\in\mathsf{Plan}(\mu,\nu)}
\in[0,\infty],
\end{equation}
for $\mu,\nu\in\prob(X)$, where
\begin{equation}
\mathsf{Plan}(\mu,\nu)
=\set*{\pi\in\prob(X\times X) : (\mathrm p_1)_\#\pi=\mu,\ (\mathrm p_2)_\#\pi=\nu},
\end{equation}
where $\mathrm p_i\colon X\times X\to X$, $i=1,2$, are the projections on each component and $T_\#\mu\in\prob(Y)$ denotes the push-forward measure given by any $\mu$-measurable map $T\colon X\to X$.
The function $\ww_ 2$ is a distance on the \emph{Wasserstein space} 
\begin{equation}
\prob_2(X) = \set*{\mu\in\prob(X) : \int_X\dd^2(x,x_0)\di\mu(x)<\infty\ \text{for some, and thus any,}\ x_0\in X}.
\end{equation}
Note that $(\prob_2(X),\ww_2)$ is a Polish metric space which is geodesic as soon as $(X,\dd)$ is. 
In addition, letting $\mathrm{Geo}(X)$ be the set of geodesics of $(X,\dd)$, namely, curves $\gamma\colon[0,1]\to X$ such that $\dd(\gamma_s,\gamma_t)=|s-t|\,\dd(\gamma_0,\gamma_1)$, for all $s,t\in [0,1]$, any $W_2$-geodesic $\mu\colon[0,1]\to\prob_2(X)$ can be (possibly non-uniquely) represented as $\mu_t =(e_t)_\sharp \nu$ for some $\nu \in \prob(\mathrm{Geo}(X))$, where $e_t \colon \mathrm{Geo}(X)\to X$ is the evaluation map at time $t\in [0,1]$.

We endow the metric space $(X,\dd)$ with a non-negative Borel measure $\mm$ such that 
\begin{equation}
\mm\
\text{is finite on bounded sets and}\
\supp(\mm)=X.
\end{equation}
We define the \emph{(relative) entropy} functional $\ent_\mm\colon\prob_2(X)\to[-\infty,+\infty]$ by letting
\begin{equation}
\ent_\mm(\mu)= \int_{X} \rho\log \rho\di\mm
\end{equation}
if $\mu=\rho\mm$ and $\rho\log \rho \in \leb^1(X,\mm)$, while we set $\ent_\mm(\mu) = +\infty$ otherwise.

\begin{definition}[$\cd(K,\infty)$ property]
\label{def:cd}
We say that a metric-measure space $(X,\dd,\mm)$ satisfies the $\cd(K,\infty)$ property if, for any $\mu_0,\mu_1\in\prob_2(X)$ with $\ent_\mm(\mu_i)<+\infty$, $i=0,1$, there exists a $W_2$-geodesic $[0,1]\ni s\mapsto\mu_s\in\prob_2(X)$ joining them such that
\begin{equation}
\label{eq:cd}
\ent_\mm(\mu_s)
\le 
(1-s)\,\ent_\mm(\mu_0)+s\,\ent_\mm(\mu_1)
-
\frac K2\,s(1-s)\,\ww_2^2(\mu_0,\mu_1)
\end{equation} 
for every $s\in[0,1]$.
\end{definition}

The geodesic $K$-convexity of $\ent_\mm$ in~\eqref{eq:cd} can be reinforced to additionally encode an upper bound on the dimension on the space, as recalled below.
For $N \in (1,\infty)$, we let
\begin{equation}
S_N(\mu,\mm) = -\int_X \rho^{-1/N} \di \mu, \qquad \mu = \rho\mm + \mu^\perp,
\end{equation}
be the $N$-\emph{Rényi entropy} of $\mu\in\prob_2(X)$ with respect to $\mm$,
where $\mu = \rho\mm + \mu^\perp$ denotes the Radon--Nikodym decomposition of $\mu$ with respect to $\mm$. 

\begin{definition}[$\cd(K,N)$ property]
\label{def:cd_N}
We say that a metric-measure space $(X,\dd,\mm)$ satisfies the $\cd(K,N)$ property for some $N\in[1,\infty)$ if, for any $\mu_0,\mu_1\in\prob_2(X)$ with $\mu_i=\rho_i\mm$, $i=0,1$, there exists a $W_2$-geodesic $[0,1]\ni s\mapsto\mu_s\in\prob_2(X)$ joining them, with $\mu_s = (e_s)_\sharp \nu$ for some $\nu \in \prob(\mathrm{Geo}(X))$ such that
\begin{equation}
S_{N'}(\mu_s,\mm)
\le 
-\int_{\mathrm{Geo}(X)}\left[\tau_{K,N'}^{(1-s)}(\dd(\gamma_0,\gamma_1))\rho_0^{-1/N'}(\gamma_0) +\tau_{K,N'}^{(s)}(\dd(\gamma_0,\gamma_1))\rho_1^{-1/N'}(\gamma_1)\right]\di\nu(\gamma)
\end{equation}
for every $s\in[0,1]$, $N'\geq N$.
Here $\tau_{K,N}^{(s)}$ is the \emph{model distortion coefficient}, see \cite{S06-II}*{p.~137}.
\end{definition}

\begin{remark} The $\cd(0,N)$ corresponds to the convexity of the $N'$-Rényi entropy
\begin{equation}
S_{N'}(\mu_s,\mm)
\le (1-s) S_{N'}(\mu_0,\mm) +s S_{N'}(\mu_1,\mm),
\end{equation}
for every $s\in[0,1]$ and $N'\geq N$, with $\mu_0,\mu_1\in\prob_2(X)$ as in \cref{def:cd_N}.
\end{remark}

\begin{remark}
For a $\cd(K,N)$ metric-measure space, $K$ and $N$ represent a lower bound on the Ricci tensor and an upper bond on the dimension, respectively, and we have 
\begin{align}
\cd(K,N) \implies  \cd(K,N') &  \qquad \text{for all}\ N'\ge N,\ N,N'\in[1,\infty], \\
\cd(K,N) \implies  \cd(K',N) & \qquad \text{for all}\
K'\le K,\ K,K'\in\R.
\end{align}
In particular, the $\cd(K,\infty)$ condition \eqref{eq:cd} is the weakest of all the curvature-dimension conditions for fixed $K \in \R$.
\end{remark}

\section{Proofs}
\label{sec:proofs}

We first deal with Theorems~\ref{res:space_i} and \ref{res:commutativity}, from which \cref{res:no-BE} immediately follows.

\subsection{Proof of \texorpdfstring{\cref{res:space_i}}{Theorem 1.4}} 
\label{sec:proof_space_i}

We divide the proof in four steps.

\subsubsection*{Step 1: passing to the nilpotent approximation via blow-up}

Let $(\R^n,\widehat\vfam\,)$ be the nilpotent approximation of $(M,\vfam)$ at some fixed point $o\in M$ as explained in \cref{subsec:nilpotent_approx}. 
Let $u\in C^\infty_c(M)$ and, without loss of generality, let us assume that $\supp u$ is contained in the domain of the privileged coordinates at $o\in M$.
In particular, we identify $u$ with a $C^\infty_c$ function on $\R^n$.
We now  apply~\eqref{eq:be} to the dilated function
\begin{equation}
u_\eps = u \circ \dil_{1/\eps}\in C^\infty_c(\R^n),
\quad
\text{for}\ \eps>0,
\end{equation}  
and evaluate this expression at the point $\dil_{\eps}(x)\in\R^n$.
Exploiting the expressions in~\cref{res:formule_utili}, we get that
\begin{equation}
\label{eq:bochner_eps}
\sum_{i,j=1}^{\numb}
X_i^\eps u
\,
\big(
X_{ijj}^\eps u-X_{jji}^\eps u
\big)
-
(X_{ij}^\eps u)^2
+ 
\resto_{i,j}^\eps u
\le0,
\end{equation} 
where $X_i^\eps$ is as in~\eqref{eq:def_vec_approx} , $X_{ijk}=X_iX_jX_k$ whenever $i,j,k\in\set*{1,\dots,\numb}$, and $\resto_{i,j}^\eps$ is a  reminder locally uniformly vanishing as $\eps\to 0$.
Therefore, letting $\eps\to 0$ in~\eqref{eq:bochner_eps}, by the convergence in~\eqref{eq:def_vec_approx} we get
\begin{equation}
\label{eq:blowuppo}
\sum_{i,j=1}^{\numb}
\widehat X_i u
\,
\big(
\widehat X_{ijj} u-\widehat X_{jji} u
\big)
-
\big(\widehat X_{ij} u\big)^2
\le0,
\end{equation}
which is~\eqref{eq:be} with $K=0$ for the nilpotent approximation $(\R^n,\widehat\vfam\,)$.

\subsubsection*{Step 2: improvement via homogeneous structure}

We now show that~\eqref{eq:blowuppo} implies a stronger identity, see~\eqref{eq:blowuppo4} below, obtained from~\eqref{eq:blowuppo} by removing the squared term and replacing the inequality with an equality.
Recall, in particular, the definition of weight of (privileged) coordinates in Definition \ref{def:coord}.
We take $u\in C^\infty(\R^n)$ of the form
\begin{equation}
u = \alpha + \gamma,
\end{equation}
where $\alpha$ and $\gamma$ are homogeneous polynomial of weighted degree~$1$ and at least $3$, respectively. Since $X_I\alpha=0$ as soon as the multi-index satisfies $|I|\geq2$ (see \cite{B96}*{Prop.~4.10}), we can take the terms with lowest homogeneous degree in~\eqref{eq:blowuppo} to get
\begin{equation}\label{eq:blowuppo2}
 \sum_{i,j=1}^{\numb} \widehat{X}_i\alpha \left(\widehat{X}_{ijj}\gamma -\widehat{X}_{jji}\gamma \right) =  \sum_{i=1}^{\numb} \widehat{X}_i\alpha\,\big[\widehat{X}_i,\widehat{\Delta}\big](\gamma) \leq 0
\end{equation}
for all such~$\alpha$ and~$\gamma$. In the second equality, we used the fact that the sub-Laplacian $\widehat{\Delta}$ is a sum of squares as in \eqref{eq:laplaciannilpotent}.
Since $\alpha$ can be replaced with $-\alpha$, we must have that
\begin{equation}\label{eq:blowuppo3}
\sum_{i=1}^{\numb} \widehat{X}_i\alpha\,\big[\widehat{X}_i,\widehat{\Delta}\big](\gamma) = 0.
\end{equation}
Observing that $\widehat{X}_i\alpha$ is homogeneous of degree~$0$, and thus a constant function, we can rewrite~\eqref{eq:blowuppo3} as
\begin{equation}\label{eq:blowuppo4}
\left[\,\sum_{i=1}^{\numb} \widehat{X}_i\alpha\,\widehat{X}_i,\widehat{\Delta}\right](\gamma) =0,
\end{equation}
which is the seeked improvement of~\eqref{eq:blowuppo}.

\subsubsection*{Step 3: construction of the space $\mathfrak i\subset\mathfrak g^1$}
 
Let $\mathbb{P}_1^n$ be the vector space of homogeneous polynomials of weighted degree~$1$ on $\R^n$. 
Notice that 
\begin{equation}
\mathbb{P}_1^n = \spn\set*{z_i\mid i=1,\dots,k_1},
\quad
k_1 = \dim\distr|_0,
\end{equation}
that is, $\mathbb{P}_1^n$ is generated by the monomials given by the coordinates of lowest weight. 
We now define a linear map $ \phi \colon \mathbb{P}_1^n \to \mathfrak{g}^{1}$ by letting
\begin{equation}
\phi[\alpha] =
\widehat\nabla\alpha= 
\sum_{i=1}^{\numb} \widehat{X}_i\alpha\,\widehat{X}_i
\end{equation}
for all $\alpha\in\mathbb P_1^n$ (recall \cref{res:formule_utili}).
We claim that $\phi$ is injective. 
Indeed, if $\phi[\alpha] = 0$ for some $\alpha\in\mathbb P_1^n$, then, by applying the operator $\phi[\alpha]$ to the polynomial~$\alpha$, we get
\begin{equation}
0=\phi[\alpha](\alpha)
=
\left(
\sum_{i=1}^{\numb} \widehat{X}_i\alpha\,\widehat{X}_i
\right)
(\alpha)
=
\sum_{i=1}^{\numb} (\widehat{X}_i\alpha)^2.
\end{equation}
Thus $\widehat{X}_i\alpha=0$ for all $i=1,\dots,\numb$.
Hence $\alpha$ must have weighted degree at least $2$.
However, since $\alpha$ is homogeneous of weighted degree $1$, we conclude that $\alpha=0$, proving that $\ker\phi = \set*{0}$. 
We can thus define the subspace
\begin{equation}
\mathfrak{i}= \phi[\mathbb{P}_1^n] \subset \mathfrak{g}^{1}.
\end{equation}
By~\eqref{eq:blowuppo4}, any $\widehat{X}\in \mathfrak{i}$ is such that $[\widehat{X},\widehat{\Delta}](\gamma)=0$ for any homogeneous polynomial $\gamma$ of degree at least $3$. 
Exploiting the definitions given in Section \ref{subsec:nilpotent_approx}, we observe that a differential operator $P$, homogeneous of weighted degree $-d \in \mathbb{Z}$, has the form
\begin{equation}\label{eq:homodiff}
P = \sum_{\mu,\nu} a_{\mu,\nu} z^\mu \frac{\partial^{|\nu|}}{\partial z^{\nu}},
\end{equation}
where $\mu=(\mu_1,\dots,\mu_n)$, $\nu=(\nu_1,\dots,\nu_n)$, $\mu_i,\nu_j \in \N \cup \{0\}$, $a_{\mu,\nu}\in\R$, and the weighted degree of every addend in \eqref{eq:homodiff} is equal to $-d$, namely, $\sum_{i=1}^n (\mu_i -\nu_i)w_i=-d$.

Thus, since $\widehat{X}$ and $\widehat{\Delta}$ are homogeneous differential operators of order $-1$ and $-2$, respectively, then $[\widehat{X},\widehat{\Delta}]$ has order $-3$, see \cite{B96}*{Prop.~5.16}. It follows that $[\widehat{X},\widehat{\Delta}]=0$ as differential operator acting on $C^\infty(\R^n)$.

We now show~\eqref{eq:space_i_complement}, that is, $\mathfrak{g}^1=\mathfrak{i}\oplus \mathfrak{h}^1$.
Let us first observe that $\mathfrak{i}\cap\mathfrak{h}=\set*{0}$. 
Indeed, if $\phi[\alpha]\in \mathfrak{h}$ for some $\alpha\in\mathbb P_1^n$, that is, $\phi[\alpha]|_0 = 0$, 
then $\widehat{X}_i\alpha|_0 = 0$ for all $i=1,\dots,\numb$.
Since $\widehat{X}_i\alpha$ is a constant function, this implies $\phi[\alpha]=0$, as claimed.
Therefore, since $\dim\mathfrak{i} = \dim\mathbb{P}_1^n = k_1$, we must have $\mathfrak{g}^{1} = \mathfrak{i} \oplus \mathfrak{h}^{1}$ thanks to \cref{res:sum_space_zero} below.

\begin{lemma}
\label{res:sum_space_zero}
The subspace $\mathfrak{h}^1\subset \mathfrak{g}^1$ has codimension $k_1$.
\end{lemma}

\begin{proof}
Let $\mathfrak{v} \subset \mathfrak{g}^1$ such that $\mathfrak{g}^1=\mathfrak{v}\oplus \mathfrak{h}^1$. 
We claim that $\dim\mathfrak{v}|_0 = \dim\mathfrak{v}$, where $\dim\mathfrak{v}|_0$ is the dimension of $\mathfrak{v}|_0$ as a subspace of $T_0\R^n$, while $\dim\mathfrak{v}$ is the dimension of $\mathfrak{v}$ as a subspace of $\mathfrak{g}$. 
Indeed, we have the trivial inequality $\dim\mathfrak{v}|_0 \leq \dim\mathfrak{v}$.
On the other hand, if strict inequality holds, then $\mathfrak{v}$ must contain non-zero vector fields vanishing at zero, contradicting the fact that $\mathfrak{v}\cap\mathfrak{h}=\set*{0}$.  Therefore, since $\dim\mathfrak{g}^{1}|_0=k_1$ and $\dim\mathfrak{h}^1|_0 =0$, we get $\dim\mathfrak v=\dim\mathfrak v|_0=k_1$,  proving the lemma.
\end{proof}

\subsubsection*{Step 4: proof of the Killing property} We have so far proved the existence of $\mathfrak{i}$ such that $\mathfrak{g}^1 = \mathfrak{i}\oplus \mathfrak{h}^1$, and such that any element $Y\in \mathfrak{i}$ commutes with the sub-Laplacian $\widehat{\Delta}$. We now show that all such $Y$ is a Killing vector field.

Let $Y\in\mathfrak{i}$. Since $[Y,\widehat{\Delta}]=0$, the induced flow $\phi_s^Y$, for $s\in \R$, commutes with $\widehat{\Delta}$ when acting on smooth functions, that is,
\begin{equation}\label{eq:commutflow}
\widehat\Delta(u \circ \phi_s^Y)
=
(\widehat\Delta u) \circ \phi_s^Y
\end{equation}
for all $u\in C^\infty(\R^n)$ and $s\in\R$. Recall the sub-Riemannian Hamiltonian $\widehat{H}:T^*\R^n \to \R$,
\begin{equation}\label{eq:hamhat}
\widehat{H}(\lambda) = \frac{1}{2}\sum_{i=1}^{\numb} \langle \lambda, \widehat{X}_i\rangle^2,
\end{equation}
for all $\lambda \in T^*\R^n$. By \eqref{eq:laplaciannilpotent}, $\widehat{H}$ is the principal symbol of  $\widehat{\Delta}$. Thus, from \eqref{eq:commutflow} it follows
\begin{equation}
\widehat{H}\circ \left(\phi_s^Y\right)^* = \widehat{H},
\end{equation}
for all $s\in \R$, where the star denotes the pull-back, and thus $\left(\phi_s^Y\right)^*$ is a diffeomorphism on $T^*\R^n$. This means that $\phi_s^Y$ is an isometry, as we now show. Indeed, for any given $x\in\R^n$, the restriction $\widehat{H}|_{T^*_x\R^n}$ is a quadratic form on $T^*_x\R^n$, so $(\phi_s^Y)^*$ must preserve its kernel, that is,
\begin{equation}
\label{eq:dual_annik}
(\phi_s^Y)^* \ker \widehat{H}|_{T_{\phi_s^Y(x)}^*\R^n} = \ker \widehat{H}|_{T_{x}^*\R^n}
\end{equation}
for all $x\in\R^n$.
By \eqref{eq:hamhat}, it holds $\ker\widehat{H}|_{T_x^*\R^n} = \widehat{\distr}^\perp_x$, where $\perp$ denotes the annihilator of a vector space. By duality, from~\eqref{eq:dual_annik}  we obtain that $(\phi_s^Y)_*\widehat\distr_x=\widehat\distr_{\phi_s^Y(x)}$ for all $x\in\R^n$ as required by \eqref{eq:point-dist-pres}. Finally, for $\lambda\in T_x^*M$, let $\lambda^\#\in\distr_x$ be uniquely defined by $\g_x(\lambda^\#,V)=\scalar*{\lambda,V}_x$ for all $V\in\distr_x$, and notice that the map $\lambda \mapsto \lambda^\#$ is surjective on $\distr_x$. Then it holds $\|\lambda^\#\|^2_x = 2\widehat{H}(\lambda)$, see \cref{res:andrei}. Thus, since $(\phi_s^Y)^*$ preserves $\widehat{H}$, the map $(\phi_s^Y)_*$ preserves the sub-Riemannian norm, and thus $\g$. This means that $\phi_s^Y$ is an isometry, concluding the proof of \cref{res:space_i}. \hfill $\qed$

\subsection{Proof of \texorpdfstring{\cref{res:commutativity}}{Theorem 1.5}}
\label{sec:proof_commutativity}

We claim that
\begin{equation}
\label{eq:gjgj=hj}
\mathfrak{g}^j= \mathfrak{h}^j
\quad
\text{for all}\
j\geq 2.
\end{equation}
Note that~\eqref{eq:gjgj=hj} is enough to conclude the proof of \cref{res:commutativity}, since, from~\eqref{eq:gjgj=hj} combined with \eqref{eq:ggen} and~\eqref{eq:hgen}, we immediately get that
\begin{equation}
\mathfrak{g} =  \mathfrak{g}^1 \oplus \mathfrak{h}^2 \oplus \dots \oplus \mathfrak{h}^s.
\end{equation}
In particular, we deduce that $\mathfrak{g}|_0 = \mathfrak{g}^1|_0$, which in turn implies that $\mathfrak{g}$ must be commutative, otherwise the bracket-generating condition would fail.
To prove~\eqref{eq:gjgj=hj}, we proceed by induction on $j\ge2$ as follows.

\subsubsection*{Proof of the base case \texorpdfstring{$j=2$}{j=2}}

We begin by proving the base case $j=2$ in~\eqref{eq:gjgj=hj}.
To this aim, let $\widehat{X} \in \mathfrak{i}$ and $\widehat{Y}\in \mathfrak{g}^1$. 
By definition of Lie bracket, we can write
\begin{equation}
\big(\phi_{-s}^{\widehat X}\big)_* \widehat{Y} = s\,\big[\widehat{X},\widehat{Y}\big] + o(s)
\quad
\text{as}\ s\to 0,
\end{equation}
where $\phi_s^{\widehat X}$, for $s\in\R$, is the flow of $\widehat X$.
Since $\mathfrak{g}^1|_x = \widehat{\distr}|_x$ for all $x\in \R^n$, and since $\widehat{X}$ is Killing (in particular \eqref{eq:point-dist-pres} holds for its flow), we have that $[\widehat{X},\widehat{Y}]|_x \in \widehat{\distr}|_x$ for all $x\in \R^n$. 
Since $[\widehat{X},\widehat{Y}]\in\mathfrak{g}^2$ and so, in particular, $[\widehat{X},\widehat{Y}]$ is homogeneous of degree $-2$, we have
\begin{equation}
[\widehat{X},\widehat{Y}]|_0 = \sum_{j\,:\,w_j=2} a_j \,\partial_{z_j}|_0,
\end{equation}
for some constants $a_j\in \R$. But we also must have that $[\widehat{X},\widehat{Y}]|_0 \in \widehat{\distr}|_0$ and so, since 
\begin{equation}
\widehat{\distr}|_0 = \spn \set*{\partial_{z_j}: w_j = 1}
\end{equation}
according to \cref{def:coord}, $[\widehat{X},\widehat{Y}]|_0=0$, that is, $[\widehat{X},\widehat{Y}] \in \mathfrak{h}$. 
We thus have proved that
$[\mathfrak{i},\mathfrak{g}^1] \subset \mathfrak{h}^2$.
In particular, since $\mathfrak{g}^1 = \mathfrak{i} \oplus \mathfrak{h}^1$, we get 
\begin{equation}
\label{eq:abba}
[\mathfrak i,\mathfrak i]\subset\mathfrak h^2
\quad
\text{and}
\quad
[\mathfrak i,\mathfrak h^1]\subset\mathfrak h^2,
\end{equation}
from which we readily deduce~\eqref{eq:gjgj=hj} for $j=2$.

\subsubsection*{Proof of the induction step}
Let us assume that~\eqref{eq:gjgj=hj} holds for some $j\in\N$, $j\ge2$.
Since $\mathfrak g^1=\mathfrak i\oplus\mathfrak h^1$, by the induction hypothesis we can write
\begin{equation}
\mathfrak g^{j+1}
=
[\mathfrak g^1,\mathfrak g^j]
=
[\mathfrak g^1,\mathfrak h^j]
=
[\mathfrak i,\mathfrak h^j]
+
[\mathfrak h^1,\mathfrak h^j]
=
[\mathfrak i,\mathfrak h^j]
+
\mathfrak h^{j+1}.
\end{equation}
We thus just need to show that $[\mathfrak i,\mathfrak h^j]\subset\mathfrak h^{j+1}$ for all $j\in\N$ with $j\ge2$. 
Note that we actually already proved the case $j=1$ in~\eqref{eq:abba}.
Again arguing by induction (taking $j=1$ as base case), by the Jacobi identity and~\eqref{eq:abba} we have
\begin{equation}
[\mathfrak i,\mathfrak h^{j+1}]
=
[\mathfrak i,[\mathfrak h^1,\mathfrak h^j]]
=
[\mathfrak h^1,[\mathfrak h^j,\mathfrak i]]
+
[\mathfrak h^j,[\mathfrak i,\mathfrak h^1]]
\subset
[\mathfrak h^1,\mathfrak h^{j+1}]
+
[\mathfrak h^j,\mathfrak h^2]
=
\mathfrak h^{j+2}
\end{equation}
as desired, concluding the proof of the induction step. \hfill $\qed$

\begin{remark}[Proof of \cref{res:commutativity} in the case $\mathfrak{h}=\set*{0}$]
The proof of \cref{res:commutativity} is much simpler if the nilpotent approximation $(\R^n,\widehat\vfam)$ is a Carnot group, i.e., $\mathfrak{h}=\set*{0}$. 
Indeed, in this case, the base case $j=2$ in~\eqref{eq:gjgj=hj} immediately implies that $\mathfrak{g}^2=\mathfrak h^2=\set*{0}$, which in turn gives $\mathfrak{g}=\mathfrak{g}^1$, so that $\mathfrak g$ is commutative.
\end{remark}

\subsection{Proof of \texorpdfstring{\cref{res:R}}{Theorem 1.8}}
\label{sec:proof_R}

In the following, we assume that the reader is familiar with the notions of \emph{upper gradient} and of \emph{$q$-upper gradient}, see~\cite{AGS13} for the precise definitions. 
The next two lemmas are proved in \cite{HK00} for sub-Riemannians structures on $\R^n$ equipped with the Lebesgue measure, and are immediately extended to the weighted case.

\begin{lemma}
\label{res:ug}
Let $(M,\dd,\mm)$ be as in \cref{res:R}. 
If $u \in C(M)$ and $0\leq g \in \leb^1_{\mathrm{loc}}(M,\mm)$ be an upper gradient of $u$, then $u\in\sobh^{1,1}_{\rm loc}(M,\mm)$ with $\|\nabla u\|\le g$ $\mm$-a.e. In particular, if $u \in \Lip(M,\dd)$, then $\|\nabla u\| \leq \Lip(u)$.
\end{lemma}

\begin{proof}
Without loss of generality we may assume that $M=\Omega\subset\R^n$ is a bounded open set, the sub-Riemannian structure is induced by a family of smooth bracket-generating vector fields $\vfam=\set*{X_1,\dots,X_{\numb}}$ on $\Omega$ and $\mm=\theta\mathscr{L}^n$, where $\theta \colon \Omega \to [0,\infty)$ is smooth and satisfies $0<\inf_{\Omega}\theta \leq \sup_{\Omega}\theta < \infty$. 
Hence, $\leb^1(\Omega,\theta \mathscr{L}^n) = \leb^1(\Omega, \mathscr{L}^n)$ as sets, with equivalent norms, so that $0\le g \in \leb^1_{\mathrm{loc}}(\Omega,\mathscr{L}^n)$ is an upper gradient of $u \in C(\Omega)$.
Hence,    
by~\cite{HK00}*{Thm.~11.7}, we get that $u\in \sobh^{1,1}_{\rm loc}(\Omega,\mathscr L^n)$, with $\|\nabla u\|\leq g$ $\mathscr{L}^n$-a.e., and thus $\theta\mathscr{L}^n$-a.e., on~$\Omega$.
By definition of distributional derivative, we can write
\begin{equation}
\int_{\Omega} v\,  X_i u \di x = \int_{\Omega} u \left[-X_iv + \diverg(X_i) v\right]\di x, \quad \forall\,v\in C^1_c(\Omega),\ i=1,\dots,\numb,
\end{equation}
where $\diverg$ denotes the Euclidean divergence. We apply the above formula with test function $v=\theta w$, for any $w\in C^1_c(\Omega)$, getting
\begin{equation}
\int_{\Omega} w \, X_i u \,\theta \di x = \int_{\Omega} u \left[-X_iw + \diverg(X_i) w+ \frac{X_i\theta}{\theta}w\right]\theta \di x, \quad \forall\,w\in C^1_c(\Omega),\ i=1,\dots,\numb.
\end{equation}
The function within square brackets is the adjoint $X_i^*w$ with respect to the measure $\theta \mathscr{L}^n$. 
It follows that $\sobh^{1,q}(\Omega,\theta\mathscr{L}^n)=\sobh^{1,q}(\Omega,\mathscr{L}^n)$ as sets, with equivalent norms. 
In particular, $u\in \sob^{1,1}_{\distr,\rm loc}(\Omega,\theta \mathscr L^n)$ as desired.
\end{proof}

\begin{lemma}[Meyers–Serrin]
\label{res:ms}
Let $(M,\dd,\mm)$ be as in \cref{res:R} and let $q\in[1,\infty)$. 
Then $\sobh^{1,q}(M,\mm)\cap C^\infty(M)$ is dense in $\sobh^{1,q}(M,\mm)$.
\end{lemma}

\begin{proof}
Up to a partition of unity and exhaustion argument, we can reduce to the case $M=\Omega \subset \R^n$ is a bounded open set and $\mm = \theta \mathscr{L}^n$, where $\theta\colon \Omega \to [0,\infty)$ is as in the previous proof, so that $\sobh^{1,q}(\Omega,\mathscr{L}^n)= \sobh^{1,q}(\Omega,\theta\mathscr{L}^n)$ as sets, with equivalent norms. In particular, we can assume that $\theta\equiv 1$. This case is proved in~\cite{HK00}*{Thm.~11.9}.
\end{proof}

\begin{lemma}
\label{res:wug}
Let $(M,\dd,\mm)$ be as in \cref{res:R} and let $q\in[1,\infty)$.
If $u\in\sobh^{1,q}(M,\mm)$, then $\|\nabla u\|$ is the minimal $q$-upper gradient of $u$.
\end{lemma}

\begin{proof}
Let us first prove that $\|\nabla u\|$ is a $q$-upper gradient of $u$. 
Indeed, by \cref{res:ms}, we can find $(u_k)_{k\in\N}\subset\sobh^{1,q}(M,\mm)\cap C^\infty(M)$ such that $u_k\to u$ in $\sobh^{1,q}(M,\mm)$ as $k\to\infty$.
It is well-known that the sub-Riemannian norm of the gradient of a smooth function is an upper gradient, see~\cite{HK00}*{Prop.~11.6}. Thus, for $u_k$ it holds
 \begin{equation}
|u_k(\gamma(1))-u_k(\gamma(0))|
\le 
\int_\gamma\|\nabla u_k\|\di s.
\end{equation}
Arguing as in~\cite{H07}*{p.~179}, using Fuglede's lemma (see~\cite{H07}*{Lem.~7.5 and Sec.~10}), we pass to the limit for $k\to \infty$ in the previous equality, outside a $q$-exceptional family of curves. This proves that any Borel representative of $\|\nabla u\|$ is a $q$-upper gradient of~$u$.

We now prove that $\|\nabla u\|$ is indeed minimal. Let $0\le g\in\leb^q(M,\mm)$ be any $q$-upper gradient of~$u$.
Arguing as in~\cite{H07}*{p.~194}, we can find a sequence $(g_k)_{k\in\N}\subset\leb^q(M,\mm)$ of upper gradients of $u$ such that $g_k\ge g$ for all $k\in\N$ and $g_k\to g$ both pointwise $\mm$-a.e.\ on~$M$ and in $\leb^q(M,\mm)$ as $k\to\infty$. By \cref{res:ug}, we thus must have that $\|\nabla u\|\le g_k$ $\mm$-a.e.\ on~$M$ for all $k\in\N$. 
Hence, passing to the limit, we conclude that $\|\nabla u\|\le g$ $\mm$-a.e.\ on~$M$ for any $q$-upper gradient $g$, concluding the proof.
\end{proof}

We are now ready to deal with the proof of \cref{res:R}.

\subsection*{Proof of \texorpdfstring{\eqref{item:sob=sob}}{(i)}}
Recall that, here, $q>1$. We begin by claiming that
\begin{equation}
\label{eq:hs_inclusion_sob}
\sob^{1,q}(M,\dd,\mm)\subset\sobh^{1,q}(M,\mm)
\end{equation}
isometrically, with $\|\nabla u\| = |\slope u|_{w,q}$.
Indeed, let $u\in\sob^{1,q}(M,\dd,\mm)$.
By a well-known approximation argument, combining~\cite{AGS13}*{Prop.~4.3, Thm.~5.3 and Thm.~7.4}, we find $(u_k)_{k\in\N}\subset\Lip(M,\dd)\cap\sob^{1,q}(M,\dd,\mm)$ such that
\begin{equation}
\label{eq:gigli_approx}
u_k\to u
\quad
\text{and}
\quad
|\slope u_k|_{w,q} \to|\slope u|_{w,q}
\quad
\text{in}\ \leb^q(M,\mm). 
\end{equation}
Since $u_k\in \Lip(M,\dd)$, by \cref{res:ug} we know that $u_k\in \sobh^{1,q}(M,\mm)$. 
Hence,  by \cref{res:wug}, $|\slope u_k|_{w,q} = \|\nabla u_k\|$, and we immediately get that
\begin{equation}
\sup_{k\in\N}
\int_M\|\nabla u_k\|^q\di\mm
<\infty.
\end{equation}
Therefore, up to passing to a subsequence, $(X_i u_k)_{k\in\N}$ is weakly convergent in $\leb^q(M,\mm)$, say $X_i u_k \rightharpoonup \alpha_i \in \leb^q(M,\mm)$, for all $i=1,\dots,\numb$. We thus get that $u\in \sobh^{1,q}(M,\mm)$ with $X_i u = \alpha_i$ and thus $\nabla u = \sum_{i=1}^{\numb} \alpha_i X_i$. By stability of $q$-upper gradients, \cite{AGS13}*{Thm.~5.3 and Thm.~7.4}, $\|\nabla u\|$ is a $q$-upper gradient of $u$.  By semi-continuity of the norm, we obtain
\begin{equation}\label{eq:liminf}
\int_M \| \nabla u\|^q \di \mm \leq \liminf_{k\to \infty} \int_M \|\nabla u_k\|^q \di \mm = \int_M |\slope u|_{w,q}^q\di\mm,
\end{equation}
where we used \eqref{eq:gigli_approx}. By definition of minimal $q$-upper gradient we thus get that $\|\nabla u\| = |\slope u|_{w,q}$ $\mm$-a.e., and the claimed inclusion in~\eqref{eq:hs_inclusion_sob} immediately follows.

We now observe that it also holds
\begin{equation}
\label{eq:smooth_ok}
\sobh^{1,q}(M,\mm)\cap C^\infty(M)
\subset
\sob^{1,q}(M,\dd,\mm),
\end{equation}
with $\|\nabla u \| = |\slope u|_{w,q}$. We just need to notice that, if $u\in C^\infty(M)$, then $\|\nabla u\|$ is an upper gradient of $u$, see~\cite{HK00}*{Prop.~11.6}.
Therefore, by \cref{res:ug}, $\|\nabla u\|$ must coincide with the minimal $q$-upper gradient of $u$, i.e., $\|\nabla u\|=|\slope u|_w$ $\mm$-a.e., and~\eqref{eq:smooth_ok} readily follows.
In view of the isometric inclusions \eqref{eq:hs_inclusion_sob} and \eqref{eq:smooth_ok}, and of the density provided by \cref{res:ms}, this concludes the proof of~\eqref{item:sob=sob}. \hfill\qed

\subsection*{Proof of \texorpdfstring{\eqref{item:bochner}}{(ii)}}

Let us assume that $(M,\dd,\mm)$ satisfies the $\cd(K,\infty)$ property for some $K\in\R$.
By the previous point~\eqref{item:sob=sob}, we know that $(M,\dd,\mm)$ satisfies the $\mathsf{RCD}(K,\infty)$ property.
Consequently, since clearly $C^\infty_c(M)\subset\sob^{1,2}(M,\dd,\mm)$ by~\eqref{eq:smooth_ok}, \cite{AGS14-m}*{Rem.~6.3} (even if the measure $\mm$ is $\sigma$-finite, see~\cite{AGMR15}*{Sec.~7} for a discussion) implies that
\begin{equation}
\frac12\int_M\Delta v\,\|\nabla u\|^2\di \mm
-
\int_M v\,\g(\nabla u,\nabla\Delta u)\di \mm
\ge 
K
\int_M v\,\|\nabla u\|^2\di \mm
\end{equation} 
for all $u,v\in C^\infty_c(M)$ with $v\ge0$ on~$M$, from which we readily deduce~\eqref{eq:be}. 
\hfill\qed

\begin{remark}\label{rmk:moregeneralmeasures}
The above proofs work for more general measures $\mm$. Namely, we can assume that, locally on any bounded coordinate neighborhood $\Omega \subset \R^n$, $\mm = \theta \mathscr{L}^n$ with $\theta \in \sob^{1,1}(\Omega,\mathscr{L}^n) \cap \leb^\infty(\Omega,\mathscr{L}^n)$. In this case, the positivity of $\mm$ corresponds to the requirement that $\theta$ is locally essentially bounded from below away from zero, in charts.
\end{remark}

\subsection{Proof of \texorpdfstring{\cref{res:fullgrushin}}{Theorem 1.10}}
\label{sec:proof_fullgrushin}

We prove the two points in the statement separately.

\subsection*{Proof of \texorpdfstring{\eqref{res:juillet}}{(i)}}

The case $p=0$ has been already considered by Juillet in~\cite{J-proceeding}.
For $p>0$, we can argue as follows.
Let $A_0=[-\ell-1,-\ell]\times[0,1]$ and $A_1=[\ell,\ell+1]\times[0,1]$ for $\ell>0$. We will shortly prove that the \emph{midpoint set}
\begin{equation}
A_{1/2}
=
\set*{q\in\R^2 : \exists\, q_0\in A_0,\ \exists\, q_1\in A_1\ \text{with}\ \dd(q,q_0)=\dd(q,q_1)=\frac12\,\dd(q_0,q_1)}
\end{equation}
satisfies 
\begin{equation}
\label{eq:mid_set}
A_{1/2}\subset[-1-\eps_\ell,1+\eps_\ell]\times[0,1]
\end{equation}
for some $\eps_\ell>0$, with $\eps_\ell\downarrow 0$ as $\ell\to\infty$.
Since $\mm_p(A_0)=\mm_p(A_1)\sim \ell^p$ as $\ell\to\infty$, we get
\begin{equation}
\sqrt{\mm_p(A_0)\,\mm_p(A_1)}>\mm_p(A_{1/2})
\end{equation}
for large $\ell>0$.
This contradicts the logarithmic Brunn--Minkowski $\bm(0,\infty)$ inequality, which is a consequence of the $\cd(0,\infty)$ condition, see~\cite{V09}*{Thm.~30.7}.

To prove~\eqref{eq:mid_set}, let $q_i\in A_i$, $q_i=(x_i,y_i)$, and let $\gamma(t)=(x(t),y(t))$, $t\in[0,1]$, be a geodesic such that $\gamma(i)=q_i$, with $i=0,1$.
We first note that
\begin{equation}
\label{eq:stay_in_strip}
\min\set*{y_0,y_1}
\le 
y(t)
\le 
\max\set*{y_0,y_1}
\quad
\text{for all}\ t\in[0,1],
\end{equation}
since any curve that violates~\eqref{eq:stay_in_strip} can be replaced by a strictly shorter one satisfying~\eqref{eq:stay_in_strip}.
In particular, we get that $A_{1/2}\subset\R\times[0,1]$.
Let us now observe that
\begin{equation}
\label{eq:stima}
|x_a-x_b|\le\dd(a,b)\le|x_a-x_b|+\frac{|y_a-y_b|}{\max\set*{|x_a|,|x_b|}}
\end{equation}
for all $a=(x_a,y_a)$ and $b=(x_b,y_b)$ with $x_a,x_b\ne0$.
Therefore, if $q=(x,y)\in A_{1/2}$, then 
\begin{equation}
|x-x_0|
\le
\dd(q,q_0)
=
\frac12\,\dd(q_0,q_1)
\le 
\ell+1+ O(1/\ell)
\end{equation}
and, similarly, $|x-x_1|\le\ell+1+O(1/\ell)$.  
Since $x_0\in [-\ell-1,-\ell]$ and $x_1\in [\ell,\ell+1]$, we deduce that $|x|\le 1+O(1/\ell)$, concluding the proof of the claimed~\eqref{eq:mid_set}.
\hfill $\qed$

\subsection*{Proof of \texorpdfstring{\eqref{res:grushinbochnerae}}{(ii)}}

Out of the negligible set $\{x=0\}$, the metric $\g$ on $\mathbb{G}_p$ given by \eqref{eq:grushin_rm} is locally Riemannian.
Recalling~\eqref{eq:grushin_vol_g} and~\eqref{eq:grushin_wmeas_p}, the $\be(K,\infty)$ inequality \eqref{eq:be} is implied by the lower bound $\mathrm{Ric}_{\infty,V}\geq K$ via Bochner's formula, where $\mathrm{Ric}_{\infty,V}$ is the $\infty$-Bakry--Émery Ricci tensor of $(\R^2,\g,e^{-V}\mathrm{vol}_{\g})$, see~\cite{V09}*{Ch.~14, Eqs.~(14.36) -- (14.51)}. 
By \cref{lem:bericci} below, we have $\mathrm{Ric}_{\infty,V}\geq 0$ for all $p\geq 1$, concluding the proof. 
\hfill $\qed$

\begin{lemma}\label{lem:bericci}
Let $p\in \R$ and $N>2$. The $N$-Bakry--Émery Ricci tensor of the Grushin metric \eqref{eq:grushin_rm}, with weighted measure $\mm_p = |x|^p\di x \di y$, for all $x\neq 0$ is
\begin{equation}
\mathrm{Ric}_{N,V} = \frac{p-1}{x^2}\g -\frac{(p+1)^2}{N-2}\frac{\di x\otimes \di x}{x^2},
\end{equation}
with the convention that $1/\infty=0$.
\end{lemma}

\begin{proof}
The $N$-Bakry--Émery Ricci tensor of a $n$-dimensional weighted Riemannian structure $(\g,e^{-V}\mathrm{vol}_g)$, for $N>n$, is given by
\begin{equation}
\label{eq:fromvillani}
\mathrm{Ric}_{N,V} = \mathrm{Ric}_{\g} +\mathrm{Hess}_{\g}V - \frac{\di V\otimes \di V}{N-n},
\end{equation}
see~\cite{V09}*{Eq.~(14.36)}.
In terms of the frame~\eqref{eq:grushin_frame}, the Levi-Civita connection is given by
\begin{equation}
\nabla_X X = \nabla_X Y=0,\quad \nabla_YX = -\frac{1}{x}Y, \quad \nabla_Y Y = \frac{1}{x}X,
\end{equation}
whenever $x\neq 0$. Recalling that, from \eqref{eq:grushin_wmeas_p}, $V(x)=-(p+1)\log|x|$, for $x\neq 0$, we obtain
\begin{equation}
\label{eq:tenso_esplicit}
\mathrm{Ric}_{\g} = -\frac{2}{x^2}\g,\quad \mathrm{Hess}_{\g}V = \frac{(p+1)}{x^2}\g, \quad \di V =-\frac{p+1}{x} \di x,
\end{equation}
whenever $x\neq 0$.
The conclusion thus follows by inserting~\eqref{eq:tenso_esplicit} into \eqref{eq:fromvillani}.\end{proof}

\subsection{Proof of \texorpdfstring{\cref{res:halfgrushinCD}}{Theorem 1.11}}
\label{sec:proof_halfgrushinCD}

The statement is a consequence of the geodesic convexity of $\mathbb{G}_p^+$ and the computation of the $N$-Bakry--Émery curvature in \cref{lem:bericci}. 
Since the proof uses quite standard arguments, we simply sketch its main steps.

The interior of $\mathbb{G}_p^+$, i.e., the open half-plane, can be regarded as a (non-complete) weighted Riemannian manifold with metric $\g$ as in \eqref{eq:grushin_rm} and weighted volume as in~\eqref{eq:grushin_wmeas_p}. 
Let $\mu_0,\mu_1\in\prob_2(\mathbb{G}_p^+)$, $\mu_0,\mu_1\ll\mm_p$, with bounded support contained in the Riemannian region $\{x>\eps\}$, for some $\varepsilon\geq 0$. 

Let $(\mu_s)_{s\in[0,1]}$ be a $W_2$-geodesic joining $\mu_0$ and $\mu_1$. 
By a well-known representation theorem (see~\cite{V09}*{Cor.~7.22}), there exists $\nu\in \prob(\mathrm{Geo(\mathbb{G}_p^+)})$, supported on the set $\Gamma=(e_0\times e_1)^{-1}(\supp\mu_0\times \supp\mu_1)$, such that $\mu_s = (e_s)_\sharp \nu$ for all $s\in[0,1]$. 
Since the set $\{x\geq \eps\}$ is a geodesically convex subset of the full Grushin plane $\mathbb{G}_p$ (by the same argument of \cite{R18-gluing}*{Prop.~5}), any $\gamma \in \Gamma$ is contained for all times in the region $\{x>0\}$.
Therefore, $\Gamma$ is a set of Riemannian geodesics contained in the weighted Riemannian structure $(\set*{x>0},\g,e^{-V}\mathrm{vol}_{\g})$.  
By \cref{lem:bericci}, we have $\mathrm{Ric}_{N,V} \geq 0$ for all $N\geq N_p$, where $N_p$ is as in~\eqref{eq:halfgrushin_Np}.
At this point, a standard argument shows that the Rényi entropy is convex along Wasserstein geodesics joining $\mu_0$ with $\mu_1$, see the proof of \cite{S06-II}*{Thm.~1.7} for example.

The extension to $\mu_0,\mu_1\in\prob_2(\mathbb G_p^+)$, with $\mu_0,\mu_1\ll \mm_p$ and compact support possibly touching the singular region $\set{x=0}$, is achieved via a standard approximation argument.
More precisely, one reduces to the previous case and  exploits the stability of optimal transport~\cite{V09}*{Thm.~28.9} and the lower semi-continuity of the Rényi entropy~\cite{V09}*{Thm.~29.20}. 

Finally, the extension to general $\mu_0,\mu_1\in\prob_2(\mathbb G_p^+)$ follows the routine argument outlined in~\cite{BS-tensor}*{Rem.~2.12}, which works when $\mu_s = (e_s)_\sharp \nu$, $s\in[0,1]$, and $\nu$ is concentrated on a set of non-branching geodesics. This proves the `if' part of the statement.

The `only if' part is also standard. The $\cd(0,N)$ condition for $N>2$ implies that, on the Riemannian region $\{x>0\}$, $\mathrm{Ric}_{N,V} \geq 0$, but this is false for $N<N_p$. 

The fact that $\mathbb{G}_{p}^+$ is infinitesimally Hilbertian follows from  \cref{rmk:inf_hilb_local}, by noting that $\mm_p$ is positive and smooth out of the closed set $\{x=0\}$, which has zero measure. An alternative proof follows from the observation that $\mathbb{G}_p^+$ is a Ricci limit, see~\cite{PW22}.
\hfill\qed

\appendix

\section{Gradient and Laplacian representations formulas}
\label{sec:formulas}

For the reader's convenience, in this appendix we provide a short proof of the representation formulas~\eqref{eq:gradient_representation} and~\eqref{eq:laplacian_representation}, in the rank-varying case.

\begin{lemma}
\label{res:andrei}
For $\lambda\in T^*M$, let $\lambda^\#\in\distr$ be uniquely defined by
\begin{equation}
\label{eq:diesis_def}
\g(\lambda^\#,V)=\scalar*{\lambda,V}
\end{equation}
for all $V\in\distr$, where $\scalar*{\cdot,\cdot}$ denotes the action of covectors on vectors. Then
\begin{equation}
\label{eq:andrei}
\|\lambda^\#\|^2
=
\sum_{i=1}^{\numb}\scalar*{\lambda^\#,X_i}^2.
\end{equation}
As a consequence, if $\lambda,\mu\in T^*M$, then
\begin{equation}
\label{eq:andrei_polar}
\g(\lambda^\#,\mu^\#)=\sum_{i=1}^{\numb}\scalar*{\lambda,X_i}\scalar*{\mu,X_i}.
\end{equation}
\end{lemma}

\begin{proof}
Given $u\in\R^{\numb}$, we set $X_u=\sum_{i=1}^{\numb}u_iX_i$ and define
\begin{equation}
u^*\in\mathrm{argmin}\set*{v\mapsto |v| : v\in\R^{\numb},\  X_v = X_u}.
\end{equation}
In other words, for $X_u\in\distr$, $u^*$ is the element of minimal Euclidean norm such that $X_{u^*} = X_u$. 
Note that, by definition, it holds $\|X_u\| = |u^*|$. We thus have 
\begin{equation}
\|\lambda^\#\|
=
\sup\set*{\g(\lambda^\#,X) : \|X\|=1,\ X\in\distr}
=
\sup\set*{\g(\lambda^\#,X_u) : |u^*|=1,\ u\in\R^{\numb}}.
\end{equation}
We now claim that 
\begin{equation}
\label{eq:sup_uguali}
\sup\set*{\g(\lambda^\#,X_u) : |u^*|=1,\ u\in\R^{\numb}}
=
\sup\set*{\g(\lambda^\#,X_u) : |u|=1,\ u\in\R^{\numb}}.
\end{equation}
Indeed, the inequality $\leq$ in \eqref{eq:sup_uguali} is obtained by observing that $X_u=X_{u^*}$ for any $u\in \R^{\numb}$. 
To prove the inequality $\geq$ in \eqref{eq:sup_uguali}, we observe that, if $u\in\R^{\numb}$ is such that $|u|=1$ and $0<|u^*|<1$, then $v=u/|u^*|$ satisfies $|v^*|=1$ and gives
\begin{equation}\label{eq:minorazione}
\g(\lambda^\#,X_v)
>
\g(\lambda^\#,X_v)\,|u^*|
=
\g(\lambda^\#,X_u).
\end{equation}
Furthermore, if $|u|=1$  and $u^*=0$, then $X_u = 0$ so also in this case we find $v\in\R^n$ with $v^*=1$ such that \eqref{eq:minorazione} holds. 
This ends the proof of the claimed~\eqref{eq:sup_uguali}.
Hence, since 
\begin{equation}
\g(\lambda^\#,X_u)
=
\sum_{i=1}^{\numb} \g(\lambda^\#,X_i)\,u_i,
\end{equation}
we easily conclude that 
\begin{equation}
\|\lambda^\#\|
=
\sup\set*{\g(\lambda^\#,X_u) : |u|=1,\ u\in\R^{\numb}}
=
\sqrt{\sum_{i=1}^{\numb}\g(\lambda^\#,X_i)^2},
\end{equation}
proving~\eqref{eq:andrei}.
Equality~\eqref{eq:andrei_polar} then follows by polarization.
\end{proof}

\begin{corollary}
\label{res:formule_utili}
The following formulas hold:
\begin{equation}
\label{eqa:gradient_representation}
\nabla u = \sum_{i=1}^{\numb} X_iu\,X_i,
\end{equation}
\begin{equation}
\label{eqa:laplacian_representation}
\Delta u = \sum_{i=1}^{\numb} \left(X_i^2u + X_iu\,\diverg_\mm (X_i)\right),
\end{equation}
\begin{equation}
\label{eqa:gradient_prod}
\g(\nabla u,\nabla v)
=
\sum_{i=1}^{\numb}X_iu\,X_iv,
\end{equation}
for all $u,v\in C^\infty(M)$.
In particular, 
$\|\nabla u\|^2=\sum_{i=1}^{\numb}(X_iu)^2$
for all $u\in C^\infty(M)$.
\end{corollary}

\begin{proof}
We prove each formula separately.

\smallskip

\textit{Proof of~\eqref{eqa:gradient_representation}}.
Recalling the definition in~\eqref{eq:def_gradient}, we can pick $\lambda=du$ in~\eqref{eq:andrei_polar} to get
\begin{align}
\scalar*{d u,\mu^\#}
&=
\g(\nabla u,\mu^\#)
=
\sum_{i=1}^{\numb}
\scalar*{d u,X_i}
\scalar*{\mu,X_i}
\\
&=
\sum_{i=1}^{\numb}
X_iu\,
\scalar*{\mu,X_i}
=
\g\left(\mu^\#,\sum_{i=1}^{\numb}
X_iu\,X_i
\right)
\end{align}
whenever $\mu\in T_x^*M$. 
Since the map $\#\colon T_x^*M\to\distr_x$ is surjective, we immediately get~\eqref{eqa:gradient_representation}.

\smallskip

\textit{Proof of ~\eqref{eqa:laplacian_representation}}.
Recall that
\begin{equation}
\diverg_\mm(fX)=Xf+f\diverg_\mm(X)
\end{equation}
for any $f\in C^\infty(M)$ and $X\in\Gamma(TM)$.
Hence, from the definition in~\eqref{eq:def_laplacian}, we can compute
\begin{align}
\Delta u
=
\diverg_\mm (\nabla u)
=
\sum_{i=1}^{\numb}\diverg_\mm (X_iu\,X_i)
=
\sum_{i=1}^{\numb}\left(X_i^2u+X_iu\,\diverg_\mm(X_i)\right),
\end{align}
which is the desired~\eqref{eqa:laplacian_representation}.

\smallskip

\textit{Proof of~\eqref{eqa:gradient_prod}}.
Choosing $\lambda=du$ and $\mu=dv$ in~\eqref{eq:andrei_polar}, we can compute
\begin{equation}
g(\nabla u,\nabla v)
=
\sum_{i=1}^{\numb}
\scalar*{du,X_i}\,
\scalar*{dv,X_i}
=
\sum_{i=1}^{\numb}
X_iu\,
X_iv
\end{equation}
and the proof is complete. 
\end{proof}

\bibliography{biblio.bib}

\end{document}